%% file: CalculusOfCocycles.tex
\documentclass[12pt,a4paper]{article}

\usepackage[verbose]{geometry}
\input{Prelude}

\usepackage{titling}
\usepackage{fancyhdr}
\hangsecnum

\title{Cocycles in categories of fibrant objects}
\author{Zhen~Lin Low}
\date{28 September 2015}

\begin{document}

\maketitle
\footpar{Department of Pure Mathematics and Mathematical Statistics, University of Cambridge, Cambridge, UK. \textsc{E-mail address}: \texttt{Z.L.Low@dpmms.cam.ac.uk}}

\begin{abstract}
We establish that a category of fibrant objects (in the sense of Brown) admits a Dwyer--Kan homotopical calculus of right fractions. This is done using a homotopical calculus of cocycles, which is an auxiliary structure that can be defined on every category of fibrant objects. As an application, we deduce some non-abelian versions of the Verdier hypercovering theorem.
\end{abstract}

\input{Content}

\ifdraftdoc

\else
  \printbibliography
\fi

\end{document}

%% file: Prelude.tex
\usepackage[T1]{fontenc} 
\usepackage[utf8]{inputenc}

\ifdefined\ifdraftdoc

\else
  \usepackage{ifdraft}
  \makeatletter
  \let\ifdraftdoc\if@draft
  \makeatother
\fi

\usepackage{needspace}

\usepackage{multirow}
\usepackage{amsmath}
\allowdisplaybreaks[4]
\usepackage{amsthm}
\usepackage{amssymb}
\usepackage{amsfonts}

\renewcommand{\qedsymbol}{$\blacksquare$}
\newcommand{\openqedsymbol}{$\square$}
\newcommand{\exercisesymbol}{$\lozenge$}

\newcommand{\openproof}{\renewcommand{\qedsymbol}{\openqedsymbol}}

\usepackage{mathrsfs}
\DeclareMathAlphabet{\mathcal}{OMS}{cmsy}{m}{n}
\usepackage{stmaryrd}

\usepackage{mathtools}

\usepackage[british]{babel}

\usepackage[autostyle]{csquotes}
\usepackage{hyphenat}

\usepackage{etoolbox}
\usepackage{xstring}

\usepackage[shortlabels]{enumitem}

\usepackage[dvipsnames,hyperref]{xcolor}

\usepackage{tikz}
\usepackage{tikz-cd}

\usetikzlibrary{arrows, shapes.geometric}
\tikzset{baseline=(current bounding box.center)}

\let\LiningNumbers\relax

\DeclareTextFontCommand{\textrmup}{\rmfamily\mdseries\upshape}
\DeclareTextFontCommand{\textbfup}{\rmfamily\bfseries\upshape}
\DeclareTextFontCommand{\textttup}{\ttfamily\mdseries\upshape}
\DeclareTextFontCommand{\textsfup}{\sffamily\mdseries\upshape}

\setlist[enumerate]{font=\upshape}

\DeclareFontFamily{OT1}{pzc}{}
\DeclareFontShape{OT1}{pzc}{mb}{it}{
   <-> s * [1.150] pzcmi7t
}{}

\DeclareMathAlphabet{\mathpzc}{OT1}{pzc}{mb}{it}

\DeclareMathSymbol{\epsilon}{\mathord}{letters}{"22}
\DeclareMathSymbol{\phi}{\mathord}{letters}{"27}
\DeclareMathSymbol{\varepsilon}{\mathord}{letters}{"0F}
\DeclareMathSymbol{\varphi}{\mathord}{letters}{"1E}

\usepackage{microtype}

\makeatletter
\newcommand*{\noprelistbreak}{\@nobreaktrue\nopagebreak}
\makeatother

\usepackage{xspace}

\newcommand{\hairspace}{\ifmmode\mskip1mu\else\kern0.08em\fi}
\newcommand{\dash}{\unskip\hairspace ---\hairspace\hspace{0pt}\ignorespaces}

\makeatletter
\def\abbrdot{\@ifnextchar.{}{.\@\xspace}}
\makeatother
\newcommand{\etc}{etc\abbrdot}
\newcommand{\eg}{e.g\abbrdot}
\newcommand{\ie}{i.e\abbrdot}
\newcommand{\resp}{resp\abbrdot}
\newcommand{\confer}{cf\abbrdot}

\newcommand{\opcit}{op.\@ cit\abbrdot}

\newcommand{\Chap}{Ch.\@~}

\newcommand{\Sect}{\S\hairspace}
\newcommand{\Sects}{\S\S\hairspace}

\newcommand{\strong}[1]{\textbfup{#1}}

\renewcommand{\implies}{\ifmmode\Longrightarrow\else$\Rightarrow$\expandafter\xspace\fi}
\renewcommand{\iff}{\ifmmode\Longleftrightarrow\else$\Leftrightarrow$\expandafter\xspace\fi}

\makeatletter
\let\internal@prime\prime
\renewcommand{\prime}{\ifmmode\internal@prime\else$\sp{\internal@prime}$\expandafter\xspace\fi}
\makeatother

\ifdefined\nequiv

\else
  \def\nequiv{\not\equiv}
\fi

\newcommand{\blank}{\mathord{-}}
\newcommand{\pblank}{\mathord{(-)}}

\newcommand{\bcirc}{\bullet}

\newcommand{\embedinto}{\hookrightarrow}
\newcommand{\hoto}{\Rightarrow}

\makeatletter
\newcommand{\profto}{\mathrel{\mathpalette\profto@int\relax}}
\newcommand{\profto@int}[2]{\ooalign{$#1\rightarrow$\cr\hfil$#1\mapsfromchar$\hfil\cr}}
\makeatother

\DeclareMathSymbol{.}{\mathpunct}{letters}{"3A}
\DeclareMathSymbol{:}{\mathrel}{operators}{"3A}

\DeclareMathOperator{\mor}{mor}

\DeclareMathOperator{\und}{und}
\DeclareMathOperator{\weq}{weq}

\newcommand{\pbprod}{\mathbin{\square}}
\newcommand{\cotens}{\mathbin{\pitchfork}}

\DeclareMathOperator{\discr}{disc}

\newcommand{\Ho}[1][]{\mathop{\mathrm{Ho}\ifstrempty{#1}{}{\sb{#1}}}}

\newcommand{\LH}{\mathbf{L}^\mathrm{H}}

\DeclareMathOperator{\dom}{dom}
\DeclareMathOperator{\codom}{codom}

\let\prolim\varprojlim
\let\indlim\varinjlim

\newcommand{\id}{\mathrm{id}}

\AtBeginDocument{}

\newcommand*{\lr}[3]{\mathopen{}\mathclose{\left#1{#2}\right#3}}
\newcommand*{\lmr}[5]{\mathopen{}\mathclose{\left#1{#2}\,\middle#3\,{#4}\right#5}}

\newcommand*{\ordlr}[3]{\mathord{\left#1{#2}\right#3}}
\newcommand*{\argp}[1]{\ifstrempty{#1}{}{\lr({#1})}}
\newcommand*{\argptwo}[2]{\lr({#1, #2})}
\newcommand*{\argptwotwo}[4]{\lr({#1, #2; #3, #4})}
\newcommand*{\argb}[1]{\ifstrempty{#1}{}{\lr[{#1}]}}
\newcommand*{\parens}[1]{\lr({#1})}

\newcommand*{\bracket}[1]{\lr[{#1}]}

\newcommand*{\abs}[1]{\ordlr|{#1}|}
\newcommand*{\tuple}[1]{\ordlr({#1})}

\newcommand*{\prodtuple}[1]{\lr<{#1}>}

\newcommand*{\setbuilder}[2]{\mathord{\lmr\{{#1}|{#2}\}}}
\newcommand*{\seqbuilder}[2]{\mathord{\lmr({#1}|{#2})}}

\ifdefined\fracslash

\else
  \def\fracslash{\mathbin{/}}
\fi
\ifdefined\divslash

\else
  \def\divslash{\mathbin{/}}
\fi

\newcommand*{\ul}[1]{\smash{\underline{#1}}\vphantom{#1}}

\makeatletter
\newcommand*{\optargp}{\new@ifnextchar\bgroup{\argp}{}}
\newcommand*{\optargptwo}{\new@ifnextchar\bgroup{\expandafter\optargptwo@first}{}}
\newcommand*{\optargptwo@first}[1]{\new@ifnextchar\bgroup{\argptwo{#1}}{\argp{#1}}}
\newcommand*{\optargptwotwo}{\new@ifnextchar\bgroup{\expandafter\optargptwotwo@first}{}}
\newcommand*{\optargptwotwo@first}[2]{\new@ifnextchar\bgroup{\argptwotwo{#1}{#2}}{\argptwo{#1}{#2}}}

\newcommand*{\sbtwo}[2]{\sb{\lr({{#1}, {#2}})}}

\newcommand*{\optsb}{\new@ifnextchar\bgroup{\sb}{}}
\newcommand*{\optsp}{\new@ifnextchar\bgroup{\sp}{}}
\newcommand*{\optsbtwo}{\new@ifnextchar\bgroup{\expandafter\optsbtwo@first}{}}
\newcommand*{\optsbtwo@first}[1]{\new@ifnextchar\bgroup{\sbtwo{#1}}{\sb{#1}}}

\newcommand*{\optsbargptwo}[1][]{\sb{#1}\optargptwo}

\newcommand*{\optspsbargptwo}[1][]{\sp{#1}\optsbargptwo}

\newcommand*{\set}[1]{\new@ifnextchar\bgroup{\setbuilder{#1}}{\ordlr\{{#1}\}}}
\newcommand*{\seq}[1]{\new@ifnextchar\bgroup{\seqbuilder{#1}}{\tuple{#1}}}
\makeatother

\ifdefined\hash

\else
  \def\hash{\#}
\fi
\newcommand*{\card}[1]{\lr|{#1}|}
\ifdefined\powerset
  \undef\powerset
\else

\fi
\newcommand*{\powerset}[1][]{\mathscr{P}\ifstrempty{#1}{\hairspace}{\sb{#1}} \optargp}

\newcommand*{\smashsp}[2]{{#1}\sp{\smash{#2}}}

\newcommand*{\op}[1]{\smashsp{#1}{\mathrm{op}}}

\newcommand*{\intr}[1]{\smashsp{#1}{\circ}}
\newcommand*{\inv}[2][1]{{#2}^{-#1}}

\makeatletter
\newcommand*{\ctchoice}[2]{%
  \ifdef{\ct@homstyle}%
  {#2}%
  {#1}}
\newcommand*{\ctchoicetwo}[5]{%
  \ifdef{\ct@homcatstyle}%
  {#2}{%
  \ifdef{\ct@transfstyle}%
  {#3}{%
  \ifdef{\ct@catstyle}%
  {#4}{%
  \ifdef{\ct@homstyle}%
  {#5}{%
  {#1}}}}}}
  
\newcommand*{\DefineCategory}[2]{\DeclareRobustCommand{#1}{\ctchoice{\mathbf{#2}}{\mathbf{#2}}}}

\newcommand*{\cat}[1]{%
  \begingroup%
  \def\ct@catstyle{cat}%
  #1%
  \endgroup}
\newcommand*{\bicat}[1]{%
  \begingroup
  \def\ct@bicatstyle{bicat}%
  #1%
  \endgroup}
\newcommand*{\Hom}[1][]{%
  \begingroup%
  \def\ct@homstyle{hom}%
  \ifstrempty{#1}{\mathrm{Hom}}{#1}%
  \endgroup%
  \optargptwo}
\newcommand*{\HHom}[1][]{%
  \begingroup%
  \def\ct@homcatstyle{homcat}%
  \ifstrempty{#1}{\mathbf{Hom}}{#1}%
  \endgroup
  \optargptwo}
\newcommand*{\hpty}[1][]{%
  \begingroup%
  \def\ct@transfstyle{transf}%
  \ifstrempty{#1}{\mathrm{Nat}}{#1}%
  \endgroup
  \optargptwotwo}

\makeatother

\DefineCategory{\Set}{Set}

\DefineCategory{\Simplex}{\Delta}
\DefineCategory{\SSet}{sSet}
\DefineCategory{\SSSet}{ssSet}

\newcommand{\Cat}{\ctchoicetwo{\mathfrak{Cat}}{\mathbf{Fun}}{\mathbf{Fun}}{\mathbf{Cat}}{\mathrm{Fun}}}

\newcommand*{\commacat}[2]{\ordlr({#1 \mathbin{\downarrow} #2})}

\newcommand*{\overcat}[2]{\mathord{#1 \sb{\mathord{\divslash} #2}}}

\newcommand*{\nv}[1][]{\mathrm{N}\sp{#1}\optargp}
\newcommand*{\Nv}[1][]{\mathbf{N}\sp{#1}\optargp}

\newcommand*{\Barcx}[4][]{\mathrm{B}\ifstrempty{#1}{}{\sb{#1}}\argp{{#2}, {#3}, {#4}}}

\newcommand{\totalR}{\mathbf{R}}

\newcommand*{\Cocyc}[1][]{\ifstrempty{#1}{{\mathbb{H}}^{{\sim}{\to}}}{{#1}^{\bracket{-1 ; 1}}}\optsbargptwo}

\newcommand*{\homs}{\mathpzc{h}\optsb}

\newcommand*{\RHom}[1][]{\ifstrempty{#1}{\totalR \mathrm{Hom}}{\totalR \mathrm{Hom}_{#1}}\optargptwo}

\newcommand*{\xHom}[3][]{\ifstrempty{#1}{\ordlr[{{#2}, {#3}}]}{\ordlr[{{#2}, {#3}}]\sb{#1}}}
\newcommand*{\ulHom}[1][]{\ifstrempty{#1}{\ul{\Hom}}{\ul{\Hom[#1]}}\optargptwo}

\newcommand*{\Ext}{\mathrm{Ext}\optspsbargptwo}

\newcommand*{\FCorr}[1][]{\ifstrempty{#1}{\mathbf{FCor}}{{#1}^\mathrm{fun}}\optargptwo}

\newcommand*{\Func}[2]{\xHom{#1}{#2}}
\newcommand*{\RelFun}[2]{\xHom[\mathrm{h}]{#1}{#2}}
\newcommand*{\ZzCat}[2]{{#2}^{#1}}

\newcommand*{\wcolim}[2][]{\mathchoice%
  {\mathop{{\varinjlim}^{#2}}_{{#1}\phantom{#2}}}%
  {\varinjlim_{#1}^{#2}}%
  {\varinjlim_{#1}^{#2}}%
  {\varinjlim_{#1}^{#2}}}

\newcommand{\hocolim}[1][]{\mathop{\mathrm{ho}{\varinjlim}}_{#1}}

\newcommand{\oplaxcolim}[1][]{\wcolim[#1]{\mathrm{Gr}}}

\newcommand{\Path}{\mathrm{Path} \optargp}

\makeatletter
\@ifclassloaded{amsart}{}{ %
\renewcommand\footnoterule{%
  \vfill
  \kern-3\p@
  \hrule\@width.4\columnwidth
  \kern2.6\p@}

\renewcommand*{\@makefnmark}{\hbox{\textsuperscript{\normalfont [\LiningNumbers\@thefnmark]}}}
\renewcommand*{\@makefntext}[1]{\parindent 1.0em\noindent\ifdefempty{\@thefnmark}{}{\hb@xt@-1.0em{\hss \normalfont [\LiningNumbers \@thefnmark]}\hspace{1.0em}}#1}
\@addtoreset{footnote}{chapter}
\newcommand{\footpar}[1]{\gdef\@thefnmark{}\@footnotetext{#1}}

\newcommand{\hangsecnum}{\def\@seccntformat##1{\llap{\csname the##1\endcsname\quad}}}
}
\makeatother

\usepackage[natbib,safeinputenc,backend=biber,style=authoryear-comp,bibstyle=authoryear]{biblatex}

\DeclareSortingScheme{mysort}{
  \sort{
    \field{presort}
  }
  \sort[final]{
    \field{sortkey}
  }
  \sort{
    \name{sortname}
    \name{author}
    \name{editor}
    \name{translator}
    \field{sorttitle}
    \field{title}
  }
  \sort{
    \field{sortyear}
    \field{year}
  }
  \sort{
    \field{month}
  }
  \sort{
    \field{day}
  }
  \sort{
    \field{journal}
  }
  \sort{
    \field[padside=left,padwidth=4,padchar=0]{volume}
    \literal{0000}
  }
  \sort{
    \field{series}
  }
  \sort{
    \field[padside=left,padwidth=4,padchar=0]{number}
    \literal{0000}
  }
  \sort{
    \field{sorttitle}
    \field{title}
  }
}

\usepackage[colorlinks,linkcolor=Blue,citecolor=Blue,urlcolor=blue,pdfpagelabels,pdffitwindow=false]{hyperref}

\usepackage{thmtools}

\declaretheorem[style=plain,parent=section,title=Theorem,refname=theorem,Refname=Theorem]{thm}
\declaretheorem[style=plain,sibling=thm,title=Proposition,refname=proposition,Refname=Proposition]{prop}
\declaretheorem[style=plain,sibling=thm,title=Lemma,refname=lemma,Refname=Lemma]{lem}
\declaretheorem[style=plain,sibling=thm,title=Corollary,refname=corollary,Refname=Corollary]{cor}

\declaretheorem[style=definition,sibling=thm,title=Definition,refname=definition,Refname=Definition]{dfn}

\declaretheorem[style=definition,sibling=thm,title=Example,refname=example,Refname=Example]{example}

\declaretheoremstyle[style=remark,headfont=\scshape]{remark}
\declaretheorem[style=remark,sibling=thm,title=Remark,refname=remark,Refname=Remark]{remark}

\declaretheorem[style=plain,numbered=no,title=Theorem]{thm*}
\declaretheorem[style=plain,numbered=no,title=Proposition]{prop*}
\declaretheorem[style=remark,numbered=no,title=Remark]{remark*}
\declaretheorem[style=definition,numbered=no,title=Example]{example*}

\makeatletter
\@counteralias{numpar}{thm}

\newcommand*{\makenumpar}[1][]{%
  \par%
  \refstepcounter{numpar}%
  {\normalfont{\bfseries\textparagraph}\hspace{0.5em}\thenumpar\ifstrempty{#1}{}{ ({\normalfont #1})}.} }
\makeatother

\AtBeginDocument{%
\hyphenation{pull-back pull-backs push-out push-outs co-pro-duct co-pro-ducts co-com-plete sub-ob-ject sub-ob-jects more-over quasi-com-pact de-gen-er-ate pro-vided Gro-then-dieck Law-vere mod-ule sub-al-ge-bra fi-brant co-fi-brant fi-brant-ly co-fi-brant-ly bi-cat-e-go-ry bi-cat-e-gor-ic-al bi-cat-e-gor-ic-al-ly qua-si-cat-e-go-ry qua-si-cat-e-go-ries qua-si-cat-e-go-ric-al}%
}

\renewcommand*{\bibnamedash}{}
\newcommand*{\authoryearpunct}{\hspace*{-\bibhang}\bibsentence}

\makeatletter

\appto\citesetup{\LiningNumbers}

\renewbibmacro*{cite:label}{%
  \iffieldundef{label}
    {\printtext[bibhyperref]{\printfield[citetitle]{labeltitle}}}
    {\printtext[bibhyperref]{\printfield{label}}}}

\renewbibmacro*{date+extrayear}{%
  \iffieldundef{shorthand}
    {\iffieldundef{labelyear}
      {}
      {\printtext[brackets]{\printfield{labelyear}\printfield{extrayear}}}\addspace}
    {\printtext[brackets]{\printfield{shorthand}}}}

\renewbibmacro*{author}{%
  \ifboolexpr{
    test \ifuseauthor
    and
    not test {\ifnameundef{author}}
  }
    {\usebibmacro{bbx:dashcheck}
       {\bibnamedash}
       {\usebibmacro{bbx:savehash}%
        \printnames{author}%
    \newline\setunit{\authoryearpunct}}%
     \iffieldundef{authortype}
       {}
       {\usebibmacro{authorstrg}%
    \setunit{\addspace}}}%
    {\global\undef\bbx@lasthash
     \usebibmacro{labeltitle}%
     \setunit*{\addspace}}%
  \usebibmacro{date+extrayear}}

\renewbibmacro*{bbx:editor}[1]{%
  \ifboolexpr{
    test \ifuseeditor
    and
    not test {\ifnameundef{editor}}
  }
    {\usebibmacro{bbx:dashcheck}
       {\bibnamedash}
       {\printnames{editor}%
    \newline\setunit{\authoryearpunct}%
    \usebibmacro{bbx:savehash}}%
     \usebibmacro{#1}%
     \clearname{editor}%
     \setunit{\addspace}}%
    {\global\undef\bbx@lasthash
     \usebibmacro{labeltitle}%
     \setunit*{\addspace}}%
  \usebibmacro{date+extrayear}}

\renewbibmacro*{bbx:translator}[1]{%
  \ifboolexpr{
    test \ifusetranslator
    and
    not test {\ifnameundef{translator}}
  }
    {\usebibmacro{bbx:dashcheck}
       {\bibnamedash}
       {\printnames{translator}%
    \newline\setunit{\authoryearpunct}%
    \usebibmacro{bbx:savehash}}%
     \usebibmacro{translator+othersstrg}%
     \clearname{translator}%
     \setunit{\addspace}}%
    {\global\undef\bbx@lasthash
     \usebibmacro{labeltitle}%
     \setunit*{\addspace}}%
  \usebibmacro{date+extrayear}}

\makeatother

\appto\bibsetup{\raggedright}
\appto\bibfont{\LiningNumbers\small}

%% file: Content.tex
% !TEX root = CalculusOfCocycles.tex

\section*{Introduction}

An $\tuple{\infty, 1}$-category is like a category, except that one has hom-spaces instead of hom-sets, and the axioms hold only up to a coherent system of homotopies. In turn, a category of fibrant objects in the sense of \citet{Brown:1973} is a convenient presentation of an $\tuple{\infty, 1}$-category with finite homotopy limits. This heuristic picture can be made precise: see \citep{Szumilo:2014} and \citep{Kapulkin-Szumilo:2015}. As the name suggests, the full subcategory of fibrant objects in a model category in the sense of \citet{Quillen:1967} is indeed a category of fibrant objects \dash so \eg for any ring $R$, the category of unbounded chain complexes of left $R$-modules is a category of fibrant objects \dash but there are other examples: for instance, \citet{Uuye:2013} showed that the category of $C^*$-algebras is a category of fibrant objects, and \citet{Andersen-Grodal:1997} showed that the associated $\tuple{\infty, 1}$-category cannot be presented by a closed model category. Also, a right proper model category in which the class of weak equivalences is closed under binary product \dash such as the category of simplicial sets with the usual Kan--Quillen model structure \dash is itself a category of fibrant objects (with the same weak equivalences, but more fibrations). 

The $\tuple{\infty, 1}$-category presented by a category of fibrant objects can be constructed in various ways, some more explicit than others. The general idea is as follows: given a category $\mathcal{C}$ with a subcategory $\mathcal{W} \subseteq \mathcal{C}$ of weak equivalences, the associated $\tuple{\infty, 1}$-category is the one obtained by regarding $\mathcal{C}$ as an $\tuple{\infty, 1}$-category and freely inverting the morphisms in $\mathcal{W}$. The simplicial and hammock localisations introduced by \citet{Dwyer-Kan:1980a,Dwyer-Kan:1980b} are two ways of doing this. A third way, due to \citet{Rezk:2001}, is to take a fibrant replacement of the Rezk classification diagram with respect to the model structure for complete Segal spaces.
These constructions should be understood as enrichments of the ordinary localisation of a category with weak equivalences which capture more of the implicit ``higher-dimensional information''. For example, in the case of the category of unbounded chain complexes of abelian groups, the hom-spaces of the corresponding $\tuple{\infty, 1}$-category are equivalent to the connective truncation of the classical derived hom-complexes, so \eg the $n$-th homotopy group of the space of homotopy morphisms $A \to B$ is naturally isomorphic to the homology group $H_n \argp{\RHom[\mathbb{Z}]{A}{B}}$, also known as the hyper-ext group $\Ext[-n][\mathbb{Z}]{A}{B}$.

In applications, it is desirable to have convenient formulae for the hom-spaces of the $\tuple{\infty, 1}$-category presented by a category of fibrant objects. Although the hom-spaces of the hammock localisation already have an explicit description, just as in the case of ordinary localisation, one can sometimes obtain a simpler description when the pair $\tuple{\mathcal{C}, \mathcal{W}}$ has good properties. For instance, when $\mathcal{C}$ is a category of fibrant objects and $\mathcal{W}$ is its subcategory of weak equivalences, it is well known that every morphism in $\mathcal{C} \argb{\inv{\mathcal{W}}}$ can be represented by what \citet{Jardine:2009} calls `cocycles', \ie zigzags of the form
\[
\begin{tikzcd}
\bullet \rar[leftarrow]{\simeq} &
\bullet \rar &
\bullet
\end{tikzcd}
\]
and the main goal of this paper is to show that arbitrary zigzags in categories of fibrant objects can be reduced to cocycles in a homotopically sensitive way. More precisely:

\begin{thm*}
Any category of fibrant objects admits a homotopy calculus of right fractions in the sense of \citet{Dwyer-Kan:1980b}.
\end{thm*}

It is not so hard to verify the claim when we have \emph{functorial} path objects: indeed, this is a folklore result, and the closely related case of a Waldhausen category with a cylinder functor is described in \citep[\Sect 1]{Weiss:1999a}. One can then determine the homotopy type of the hom-spaces of the hammock localisation of a general category of fibrant objects by embedding it in a category of simplicial presheaves \dash this is the strategy employed in the proof of Proposition~3.23 in \citep{Cisinski:2010b} \dash but we will take a different approach to eliminating the hypothesis of functorial path objects. The key observation is that there is a contractible space parametrising certain special (section of trivial fibration, trivial fibration)-factorisations of weak equivalences. Almost everything else follows formally: indeed, we will treat this situation axiomatically by defining the notion of a homotopical calculus of cocycles.

As an application of the main result, we consider the category of (locally fibrant) simplicial presheaves on a site. From the point of view of homotopy theory, the central problem of sheaf theory is essentially the determination of the homotopy type the space of sections (over a given object in the site) of the hypersheaf associated with a given simplicial presheaf: for example, as \citet[\Sect 3]{Brown:1973} observed, sheaf cohomology can be paraphrased in these terms via the formula below,
\[
R^n \Gamma \argp{T, A} \cong \pi_{m-n} \totalR \Gamma \argp{T, K \argp{A, m}}
\]
where $A$ is an abelian (pre)sheaf, $K \argp{A, m}$ is the simplicial (pre)sheaf corresponding (under Dold--Kan) to the chain complex consisting of just $A$ in degree $m$, and $m \ge n$. The Verdier hypercovering theorem in its classical form is a colimit formula for sheaf cohomology in terms of generalised Čech cochain complexes, and following a suggestion of \citet{MO:answer-165324}, we derive a non-abelian version that computes (up to weak homotopy equivalence) $\totalR \Gamma \argp{T, X}$ for any locally fibrant simplicial presheaf $X$ in terms of a homotopy colimit of simplicial sets of generalised sections of $X$.

\subsection*{Outline}

\begin{itemize}
\item In \Sect 1, we collect some miscellaneous facts about homotopy colimits.

\item In \Sect 2, we review the definitions and fundamental results regarding zigzags in relative categories.

\item In \Sect 3, we introduce the notion of a homotopical calculus of cocycles, which is a sufficient condition for a category with weak equivalences to admit a homotopical calculus of right fractions.

\item In \Sect 4, we prove that a category of fibrant objects with functorial path objects admits a homotopical calculus of cocycles.

\item In \Sect 5, we define the notion of a simplicial category of fibrant objects and give a homotopy colimit formula for the hom-spaces of its hammock localisation.

\item In \Sect 6, we apply this theory to the study of simplicial presheaves on a site.

\item In \Sect A, we show that a general category of fibrant objects (\ie possibly without functorial path objects) admits a homotopical calculus of cocycles.
\end{itemize}

\subsection*{Conventions}

It will be convenient to implicitly assume that categories are small, especially in \Sects 2\hairspace --\hairspace 5 and \Sect A. Since the categories of interest are usually not small, it is not possible to apply these results as stated literally; one way to work around this to adopt a suitable universe axiom. Alternatively, because most of the categories under consideration in \Sect 6 are \emph{essentially} small, one could just replace them with small skeletons where necessary, thereby avoiding the use of universes.

\subsection*{Acknowledgements}

Discussions with Aaron Mazel-Gee led to the discovery and correction of some errors in earlier drafts of this paper. His comments also helped improve the exposition.

The author gratefully acknowledges financial support from the Cambridge Commonwealth, European and International Trust and the Department of Pure Mathematics and Mathematical Statistics.

\section{Homotopy colimits}

The following definition is due to \citet{Bousfield-Kan:1972}.

\begin{dfn}
Let $\ul{X} : \op{\ul{\mathcal{C}}} \to \ul{\cat{\SSet}}$ be a small simplicially enriched diagram. The \strong{homotopy colimit} $\hocolim[\op{\ul{\mathcal{C}}}] \ul{X}$ is the diagonal of the bisimplicial set $\Barcx[\bullet]{\ul{X}}{\ul{\mathcal{C}}}{\Delta 1}$ defined below,
\[
\Barcx[n]{\ul{X}}{\ul{\mathcal{C}}}{\Delta 1} = \coprod_{\tuple{c_0, \ldots, c_n}} X \argp{c_n} \times \ulHom[\mathcal{C}]{c_{n-1}}{c_n} \times \cdots \times \ulHom[\mathcal{C}]{c_0}{c_1}
\]
where the disjoint union is indexed over $\parens{n + 1}$-tuples of objects in $\mathcal{C}$, with the evident face and degeneracy operators.
\end{dfn}

\begin{example}
Let $X : \op{\mathcal{C}} \to \cat{\Set}$ be a small diagram and let $\mathcal{D}$ be the category of elements of $X$, \ie the \emph{opposite} of the comma category $\commacat{1}{X}$, where $1$ is a singleton set. Regarding $X$ as a diagram $\op{\mathcal{C}} \to \cat{\SSet}$, it is not hard to see that $\hocolim[\mathcal{C}] X$ is (isomorphic to) the nerve $\nv{\mathcal{D}}$.
\end{example}

We will need some miscellaneous facts about homotopy cofinality. First, let us say that a \strong{weakly contractible category} is a category $\mathcal{A}$ such that the unique morphism $\nv{\mathcal{A}} \to \Delta^0$ is a weak homotopy equivalence of simplicial sets. We then make the following definition:

\begin{dfn}
A \strong{homotopy cofinal functor} is a functor $F : \mathcal{C} \to \mathcal{D}$ with the following property: for all objects $d$ in $\mathcal{D}$, the comma category $\commacat{d}{F}$ is weakly contractible.
\end{dfn}

\begin{lem}
\label{lem:Grothendieck.fibrations.and.homotopy.cofinality}
Let $P : \mathcal{E} \to \mathcal{B}$ is a Grothendieck fibration. The following are equivalent:
\begin{enumerate}[(i)]
\item The (strict) fibres of $P$ are weakly contractible categories.

\item $P$ is a homotopy cofinal functor.
\end{enumerate}
\end{lem}
\begin{proof}
Let $b$ be an object in $\mathcal{B}$. There is a functor $\inv{P} \set{b} \embedinto \commacat{b}{P}$ sending objects $e$ in $\inv{P} \set{b}$ to $\tuple{e, \id_b}$ in $\commacat{b}{P}$, and it is well known that this functor has a right adjoint when $P : \mathcal{E} \to \mathcal{B}$ is a Grothendieck fibration. Since adjoint functors induce homotopy equivalences of nerves, it follows that $\inv{P} \set{b}$ is weakly contractible if and only if $\commacat{b}{P}$ is weakly contractible.
\end{proof}

\begin{lem}
\label{lem:homotopy.cofinal.functors.and.fully.faithful.functors}
Let $F : \mathcal{C} \to \mathcal{D}$ and $G : \mathcal{D} \to \mathcal{E}$ be functors. If $G F : \mathcal{C} \to \mathcal{E}$ is homotopy cofinal and $G : \mathcal{D} \to \mathcal{E}$ is fully faithful, then $F : \mathcal{C} \to \mathcal{D}$ is also homotopy cofinal.
\end{lem}
\begin{proof}
Let $d$ be any object in $\mathcal{D}$. If $G : \mathcal{D} \to \mathcal{E}$ is fully faithful, then there is an isomorphism $\commacat{d}{F} \cong \commacat{G \argp{d}}{G F}$, so $F : \mathcal{C} \to \mathcal{D}$ is homotopy cofinal when $G F : \mathcal{C} \to \mathcal{E}$ is.
\end{proof}

\begin{thm}[Quillen's Theorem A]
\label{thm:Quillen-A}
Homotopy cofinal functors are weak homotopy equivalences of categories.
\end{thm}
\begin{proof} \openproof
See \citep[\Sect 1]{Quillen:1973a}.
\end{proof}

\begin{lem}
\label{lem:pullback.of.left.adjoints.along.Grothendieck.fibrations}
Consider a pullback diagram in $\cat{\Cat}$:
\[
\begin{tikzcd}
\mathcal{E}' \dar[swap]{P'} \rar{L} &
\mathcal{E} \dar{P} \\
\mathcal{B}' \rar[swap]{F} &
\mathcal{B}
\end{tikzcd}
\]
If $P : \mathcal{E} \to \mathcal{B}$ is a Grothendieck fibration and $F : \mathcal{B}' \to \mathcal{B}$ has a right adjoint, then $L : \mathcal{E}' \to \mathcal{E}$ also has a right adjoint.
\end{lem}
\begin{proof}
Let $G : \mathcal{B} \to \mathcal{B}'$ be any right adjoint for $F : \mathcal{B}' \to \mathcal{B}$, let $e_1$ be an object in $\mathcal{E}$, let $e_0'$ be an object in $\mathcal{E}'$, let $b_1 = e_1$, let $b'_0 = P' \argp{e'_0}$, choose a cartesian morphism $\tilde{\epsilon}_{e_1} : \parens{\epsilon_{b_1}}^* e_1 \to e_1$ in $\mathcal{E}$ such that $P \argp{\tilde{\epsilon}_{e_1}} = \epsilon_{b_1}$, where $\epsilon_{b_1} : F \argp{G \argp{b_1}} \to b_1$ is the counit component, and let $R \argp{e_1}$ be the unique object in $\mathcal{E}'$ such that $P' \argp{R \argp{e_1}} = G b_1$ and $L \argp{R \argp{e_1}} = \parens{\epsilon_{b_1}}^* e_1$. We then have the following commutative diagram,
\[
\begin{tikzcd}[column sep=9.0ex]
\Hom[\mathcal{E}']{e'_0}{R \argp{e_1}} \dar[swap]{P'} \rar{L} &
\Hom[\mathcal{E}]{L \argp{e'_0}}{L \argp{R \argp{e_1}}} \dar{P} \rar{\Hom[\mathcal{E}]{L \argp{e'_0}}{\tilde{\epsilon}_{e_1}}} &
\Hom[\mathcal{E}]{L \argp{e'_0}}{e_1} \dar{P} \\
\Hom[\mathcal{B}']{b'_0}{G \argp{b_1}} \rar[swap]{F} &
\Hom[\mathcal{B}]{F \argp{b'_0}}{F \argp{G \argp{b_1}}} \rar[swap]{\Hom[\mathcal{B}]{F \argp{b'_0}}{\epsilon_{b_1}}} &
\Hom[\mathcal{B}]{F \argp{b'_0}}{b_1}
\end{tikzcd}
\]
where both squares and the outer rectangle are pullback diagrams; but the composite of the bottom row is a bijection, so the composite of the top row is also a bijection. Thus, $L : \mathcal{E}' \to \mathcal{E}$ indeed has a right adjoint.
\end{proof}

\begin{lem}
\label{lem:tensors.and.homotopy.cofinality}
Let $\ul{X} : \op{\ul{\mathcal{C}}} \to \ul{\cat{\SSet}}$ be a small simplicially enriched diagram. If $\ul{\mathcal{C}}$ has cotensor products $\Delta^m \cotens c$ for every standard simplex $\Delta^m$ and every object $c$ in $\mathcal{C}$, then (regarding the underlying category $\mathcal{C}$ as a simplicially enriched category with discrete hom-spaces) the canonical comparison morphism
\[
\textstyle \hocolim[\op{\mathcal{C}}] X \to \hocolim[\op{\ul{\mathcal{C}}}] \ul{X}
\]
is a weak homotopy equivalence.
\end{lem}
\begin{proof}
Let $Y_{\bullet}$ and $Y'_{\bullet}$ be the transposes of the bisimplicial sets $\Barcx[\bullet]{\ul{X}}{\ul{\mathcal{C}}}{\Delta 1}$ and $\Barcx[\bullet]{X}{\mathcal{C}}{\Delta 1}$, respectively, and for each natural number $m$, let $\mathcal{C}_m$ be the $m$-th level of the simplicial category corresponding to $\ul{\mathcal{C}}$. (Note that $\mathcal{C} = \mathcal{C}_0$.) Then,
\begin{align*}
Y_{m, n} & = \coprod_{\tuple{c_0, \ldots, c_n}} X \argp{c_n}_m \times \Hom[\mathcal{C}_m]{c_{n-1}}{c_n} \times \cdots \times \Hom[\mathcal{C}_m]{c_0}{c_1} \\
Y'_{m, n} & = \coprod_{\tuple{c_0, \ldots, c_n}} X \argp{c_n}_m \times \Hom[\mathcal{C}_0]{c_{n-1}}{c_n} \times \cdots \times \Hom[\mathcal{C}_0]{c_0}{c_1}
\end{align*}
and the canonical comparison morphism $\hocolim[\mathcal{C}] X \to \hocolim[\ul{\mathcal{C}}] \ul{X}$ is simply the diagonal of the bisimplicial set morphism $Y'_{\bullet} \to Y_{\bullet}$ defined in degree $m$ by the $m$-fold iterated degeneracy $\mathcal{C}_0 \to \mathcal{C}_m$. But, for any $c$ and $c'$ in $\mathcal{C}$,
\[
\Hom[\mathcal{C}_0]{c'}{\Delta^m \cotens c} \cong \Hom[\mathcal{C}_m]{c'}{c}
\]
so the $m$-fold iterated degeneracy $\mathcal{C}_0 \to \mathcal{C}_m$ has a right adjoint. It follows by \autoref{lem:pullback.of.left.adjoints.along.Grothendieck.fibrations} that the morphisms $Y'_m \to Y_m$ are nerves of left adjoint functors and hence are (simplicial) homotopy equivalences \emph{a fortiori}. Thus, by the homotopy invariance of diagonals,\footnote{See \eg Theorem 15.11.11 in \citep{Hirschhorn:2003}.} the induced morphism $\hocolim[\mathcal{C}] X \to \hocolim[\ul{\mathcal{C}}] \ul{X}$ is a weak homotopy equivalence.
\end{proof}

We will also need the following version of the Grothendieck construction:

\begin{dfn}
Let $\mathcal{X} : \op{\mathcal{C}} \to \cat{\Cat}$ be a small diagram. The \strong{oplax colimit} for $\mathcal{X}$ is the category $\oplaxcolim[\op{\mathcal{C}}] \mathcal{X}$ defined below:
\begin{itemize}
\item The objects are pairs $\tuple{c, x}$ where $c$ is an object in $\mathcal{C}$ and $x$ is an object in $\mathcal{X} \argp{c}$.

\item The morphisms $\tuple{c', x'} \to \tuple{c, x}$ are pairs $\tuple{f, g}$ where $f : c' \to c$ is a morphism in $\mathcal{C}$ and $g : x' \to \mathcal{X} \argp{f} \argp{x}$ is a morphism in $\mathcal{X} \argp{c'}$.

\item Composition and identities are inherited from $\mathcal{C}$ and $\mathcal{X}$.
\end{itemize}
\end{dfn}

\begin{example}
Let $X : \op{\mathcal{C}} \to \cat{\Set}$ be a small diagram. Regarding $X$ as a diagram $\op{\mathcal{C}} \to \cat{\Cat}$, it is not hard to see that $\oplaxcolim[\mathcal{C}] X$ is the category of elements of $X$, \ie $\op{\commacat{1}{X}}$.
\end{example}

\begin{thm}[Thomason's homotopy colimit theorem]
\label{thm:Thomason.hocolim}
Let $\mathcal{X} : \op{\mathcal{C}} \to \cat{\Cat}$ be a small diagram. There is a weak homotopy equivalence
\[
\textstyle \hocolim[\op{\mathcal{C}}] \nv \circ \mathcal{X} \to \nv{\oplaxcolim[\op{\mathcal{C}}] \mathcal{X}}
\]
which is moreover natural in $\mathcal{C}$ and $\mathcal{X}$.
\end{thm}
\begin{proof} \openproof
See \citep{Thomason:1979}.
\end{proof}

\section{Zigzags in relative categories}

Recall the following definitions from \citep{Barwick-Kan:2012a}:

\begin{dfn}
\ \noprelistbreak
\begin{itemize}
\item A \strong{relative category} is a pair $\mathcal{C} = \tuple{\und \mathcal{C}, \weq \mathcal{C}}$ where $\und \mathcal{C}$ is a category and $\weq \mathcal{C}$ is a (usually non-full) subcategory of $\und \mathcal{C}$ containing all the objects.

\item Given a relative category $\mathcal{C}$, a \strong{weak equivalence} in $\mathcal{C}$ is a morphism in $\weq \mathcal{C}$.

\item The \strong{homotopy category} of a relative category $\mathcal{C}$ is the category $\Ho \mathcal{C}$ obtained by freely inverting the weak equivalences in $\mathcal{C}$.

\item Given relative categories $\mathcal{C}$ and $\mathcal{D}$, a \strong{relative functor} $\mathcal{C} \to \mathcal{D}$ is a functor $\und \mathcal{C} \to \und \mathcal{D}$ that restricts to a functor $\weq \mathcal{C} \to \weq \mathcal{D}$, and the \strong{relative functor category} $\RelFun{\mathcal{C}}{\mathcal{D}}$ is the relative category whose underlying category is the full subcategory of the ordinary functor category $\Func{\und \mathcal{C}}{\und \mathcal{D}}$ spanned by the relative functors, with the weak equivalences being the natural transformations whose components are weak equivalences in $\mathcal{D}$.
\end{itemize}
\end{dfn}

\begin{remark}
The 2-category of (small) categories admits several 2-fully faithful embeddings into the 2-category of (small) relative categories; unless otherwise stated, we will regard an ordinary category as \strong{minimal relative category} where the only weak equivalences are the identity morphisms. In particular, given an ordinary category $\mathcal{C}$ and a relative category $\mathcal{D}$, we will often tacitly identify the ordinary functor category $\Func{\mathcal{C}}{\mathcal{D}}$ with the relative functor category $\RelFun{\mathcal{C}}{\mathcal{D}}$.
\end{remark}

\begin{dfn}
\ \noprelistbreak
\begin{itemize}
\item A \strong{zigzag type} is a finite sequence of non-zero integers $\tuple{k_0, \ldots, k_n}$, where $n \ge 0$, such that for $0 \le i < n$, the sign of $k_i$ is the opposite of the sign of $k_{i+1}$.

\item Given a finite sequence of integers $\seq{k_0, \ldots, k_n}$, $\bracket{k_0 ; \ldots ; k_n}$ is the relative category whose underlying category is freely generated by the graph
\[
\begin{tikzcd}
0 \rar[-] &
\cdots \rar[-] &
\abs{k_0} + \cdots + \abs{k_n}
\end{tikzcd}
\]
where (counting from the left) the first $\abs{k_0}$ arrows point rightward (\resp leftward) if $k_0 > 0$ (\resp $k_0 < 0$), the next $\abs{k_2}$ arrows point rightward (\resp leftward) if $k_1 > 0$ (\resp $k_1 < 0$), \etc, with the weak equivalences being generated by the leftward-pointing arrows.

\item A \strong{zigzag} in a relative category $\mathcal{C}$ of type $\bracket{k_0 ; \ldots ; k_n}$ is a relative functor $\bracket{k_0 ; \ldots ; k_n} \to \mathcal{C}$; given a zigzag, its \strong{length} is $m = \abs{k_0} + \cdots + \abs{k_n}$, its \strong{domain} is the image of the object $0$, and its \strong{codomain} is the image of the object $m$.
\end{itemize}
\end{dfn}

\begin{example}
For example, $\bracket{-1 ; 2}$ denotes the relative category generated by the following graph,
\[
\begin{tikzcd}
0 \rar[leftarrow]{\simeq} &
1 \rar &
2 \rar &
3
\end{tikzcd}
\]
with $1 \to 0$ being the unique non-trivial weak equivalence.
\end{example}

\begin{remark}
For any $\bracket{k_0 ; \ldots ; k_n}$, if $\abs{k_0} + \cdots + \abs{k_n} > 0$, then there is a unique zigzag type $\seq{l_0, \ldots, l_m}$ such that $\bracket{k_0 ; \ldots ; k_n} = \bracket{l_0 ; \ldots ; l_m}$. However, it is convenient to allow unnormalised notation.
\end{remark}

\begin{dfn}
Let $X$ and $Y$ be objects in a relative category $\mathcal{C}$ and let $\tuple{k_0, \ldots, k_n}$ be a finite sequence of integers. The \strong{category of zigzags} in $\mathcal{C}$ from $X$ to $Y$ of type $\tuple{k_0 ; \ldots ; k_n}$ is the category $\mathcal{C}^{\bracket{k_0 ; \ldots ; k_n}} \argp{X, Y}$ defined below:
\begin{itemize}
\item The objects are the zigzags in $\mathcal{C}$ of type $\bracket{k_0 ; \ldots ; k_n}$  whose domain is $X$ and whose codomain is $Y$.

\item The morphisms are commutative diagrams in $\mathcal{C}$ of the form
\[
\begin{tikzcd}
X \dar[equals] \rar[-] &
\bullet \dar[swap]{\simeq} \rar[-] &
\cdots \rar[-] &
\bullet \dar{\simeq} \rar[-] &
Y \dar[equals] \\
X \rar[-] &
\bullet \rar[-] &
\cdots \rar[-] &
\bullet \rar[-] &
Y
\end{tikzcd}
\]
where the top row is the domain, the bottom row is the codomain, and the vertical arrows are weak equivalences in $\mathcal{C}$.

\item Composition and identities are inherited from $\mathcal{C}$.
\end{itemize}
For brevity, we write $\mathcal{C}^{\bracket{k_0 ; \ldots ; k_n}}$ instead of $\weq \RelFun{\bracket{k_0 ; \ldots ; k_n}}{\mathcal{C}}$.
\end{dfn}

\begin{remark*}
In other words, the morphisms in $\mathcal{C}^{\bracket{k_0 ; \ldots ; k_n}} \argp{X, Y}$ are certain hammocks of width $1$, in the sense of \citet{Dwyer-Kan:1980b}.
\end{remark*}

For brevity, let us say that a \strong{weak homotopy equivalence of categories} is a functor $F : \mathcal{C} \to \mathcal{D}$ such that $\nv{F} : \nv{\mathcal{C}} \to \nv{\mathcal{D}}$ (\ie the induced morphism of nerves) is a weak homotopy equivalence of simplicial sets. The following is a variation on the homotopy calculus of right fractions introduced by \citet{Dwyer-Kan:1980b}. 

\begin{dfn}
\needspace{3\baselineskip}
A relative category $\mathcal{C}$ admits a \strong{homotopical calculus of right fractions} if it satisfies the following condition:
\begin{itemize}
\item For all natural numbers $k$ and $l$ and all objects $X$ and $Y$ in $\mathcal{C}$, the evident functor
\[
\mathcal{C}^{\bracket{-1 ; k ; l}} \argp{X, Y} \to \mathcal{C}^{\bracket{-1 ; k ; -1 ; l}} \argp{X, Y}
\]
defined by inserting an identity morphism is a weak homotopy equivalence of categories.
\end{itemize}
\end{dfn}

\begin{remark}
Let $\mathcal{C}$ be a relative category and let $\mathcal{W}$ be $\weq \mathcal{C}$ considered as a relative category where all morphisms are weak equivalences. Then the following are equivalent:
\begin{enumerate}[(i)]
\item $\mathcal{C}$ admits a homotopy calculus of right fractions in the sense of \citet{Dwyer-Kan:1980b}.

\item Both $\mathcal{C}$ and $\mathcal{W}$ admit a homotopical calculus of right fractions in the sense of the above definition.
\end{enumerate}
Moreover, if the weak equivalences in $\mathcal{C}$ have the 2-out-of-3 property, then $\mathcal{W}$ admits a homotopical calculus of right fractions if $\mathcal{C}$ does. 
\end{remark}

\begin{remark}
If a relative category $\mathcal{C}$ admits a homotopical calculus of right fractions, then $\mathcal{C}$ also admits a homotopical three-arrow calculus. In particular, the results of \citep{Low-MazelGee:2015} apply, \ie any Reedy-fibrant replacement $\widehat{\Nv{\mathcal{C}}}$ of the Rezk classification diagram $\Nv{\mathcal{C}}$ is a Segal space, and $\widehat{\Nv{\mathcal{C}}}$ is a complete Segal space if $\mathcal{C}$ is a saturated homotopical category.
\end{remark}

\begin{thm}[Dwyer and Kan]
\label{thm:fundamental.theorem.of.homotopical.two-arrow.calculi}
\needspace{3\baselineskip}
Let $\mathcal{C}$ be a relative category and let $\ul{\LH \mathcal{C}}$ be the hammock localisation. 
\begin{enumerate}[(i)]
\item If $\mathcal{C}$ admits a homotopical calculus of right fractions, then the reduction morphism $\nv{\mathcal{C}^{\bracket{-1 ; 1}} \argp{X, Y}} \to \ulHom[\LH \mathcal{C}]{X}{Y}$ is a weak homotopy equivalence of simplicial sets.

\item The reduction morphism $\nv{\mathcal{C}^{\bracket{-1 ; 1}} \argp{X, Y}} \to \ulHom[\LH \mathcal{C}]{X}{Y}$ is natural in the following sense: given any weak equivalence $X \to X'$ and any morphism $Y \to Y'$ in $\mathcal{C}$, the following diagram commutes in $\cat{\SSet}$,
\[
\begin{tikzcd}
\nv{\mathcal{C}^{\bracket{-1 ; 1}} \argp{X, Y}} \dar \rar &
\ulHom[\LH \mathcal{C}]{X}{Y} \dar \\
\nv{\mathcal{C}^{\bracket{-1 ; 1}} \argp{X', Y'}} \rar &
\ulHom[\LH \mathcal{C}]{X'}{Y'}
\end{tikzcd}
\]
where the left vertical arrow is defined by composition and the right vertical arrow is defined by concatenation.
\end{enumerate}
\end{thm}
\begin{proof} \openproof
(i). This is Proposition 6.2 in \citep{Dwyer-Kan:1980b}. Note that the second half of the `homotopy calculus of right fractions' condition is not used, so it does indeed suffice to have a homotopical calculus of right fractions.

\bigskip\noindent
(ii). Obvious.
\end{proof}

\begin{cor}
\label{cor:action.of.weak.equivalences.on.two-arrow.zigzags}
Let $\mathcal{C}$ be a relative category. If $\mathcal{C}$ admits a homotopical calculus of right fractions, then for any weak equivalences $X \to X'$ and $Y \to Y'$ in $\mathcal{C}$, the induced functor
\[
\mathcal{C}^{\bracket{-1 ; 1}} \argp{X, Y} \to \mathcal{C}^{\bracket{-1 ; 1}} \argp{X', Y'}
\]
is a weak homotopy equivalence of categories.
\end{cor}
\begin{proof} \openproof
Use naturality (as in \autoref{thm:fundamental.theorem.of.homotopical.two-arrow.calculi}) and Proposition 3.3 in \citep{Dwyer-Kan:1980b}.
\end{proof}

\section{The homotopical calculus of cocycles}

The following notion of `cocycle' is originally due to \citet{Jardine:2009}.

\begin{dfn}
\needspace{3\baselineskip}
Let $\mathcal{C}$ be a relative category and let $\mathcal{V}$ be a class of weak equivalences in $\mathcal{C}$. 
\begin{itemize}
\item Given objects $X$ and $Y$ in $\mathcal{C}$, a \strong{$\mathcal{V}$-cocycle} $\tuple{f, v} : X \profto Y$ in $\mathcal{C}$ is a diagram in $\mathcal{C}$ of the form below,
\[
\begin{tikzcd}
X \rar[leftarrow]{v} &
\tilde{X} \rar{f} &
Y
\end{tikzcd}
\]
where $v : \tilde{X} \to X$ is a morphism in $\mathcal{V}$. Given such, the \strong{domain} is $X$ and the \strong{codomain} is $Y$.

We write $\Cocyc[\mathcal{C}][\mathcal{V}]{X}{Y}$ for the full subcategory of $\mathcal{C}^{\bracket{-1 ; 1}} \argp{X, Y}$ spanned by the $\mathcal{V}$-cocycles. 

\item If $\mathcal{V} = \weq \mathcal{C}$, then we may simply say \strong{cocycle} instead of `$\mathcal{V}$-cocycle'.
\end{itemize}
\end{dfn}

\begin{remark}
In other words, a cocycle in $\mathcal{C}$ is a zigzag of type $\bracket{-1 ; 1}$.
\end{remark}
%
%\begin{prop}
%Let $\mathcal{C}$ be a relative category and let $\gamma : \mathcal{C} \to \Ho \mathcal{C}$ be the localising functor. If $\mathcal{C}$ admits a homotopical calculus of right fractions, then:
%\begin{enumerate}[(i)]
%\item Every morphism $X \to Y$ in $\Ho \mathcal{C}$ can be factored as $\gamma \argp{f} \circ \inv{\gamma \argp{w}}$ for some cocycle $\tuple{f, w} : X \profto Y$ in $\mathcal{C}$.
%
%\item Two cocycles $X \profto Y$ represent the same morphism $X \to Y$ in $\Ho \mathcal{C}$ if and only if they are in the same connected component of $\Cocyc[\mathcal{C}]{X}{Y}$.
%\end{enumerate}
%\end{prop}
%\begin{proof} \openproof
%This is an immediate consequence of Proposition 3.1 in \citep{Dwyer-Kan:1980b} and \autoref{thm:fundamental.theorem.of.homotopical.two-arrow.calculi}.
%\end{proof}

\makenumpar
Recall that a \strong{category with weak equivalences} is a relative category in which the weak equivalences have the 2-out-of-3 property and include all isomorphisms.
For the remainder of this section, $\mathcal{C}$ is a category with weak equivalences and $\mathcal{W} = \weq \mathcal{C}$.

Heuristically, a homotopical calculus of cocycles for $\mathcal{C}$ consists of three pieces of data: a class $\mathcal{V}$ of ``good'' weak equivalences in $\mathcal{C}$, a category of ``enhanced'' cocycles, and a forgetful functor from the category of ``enhanced'' cocycles to the category of cocycles in $\mathcal{C}$, such that:
\begin{itemize}
\item $\mathcal{V}$ is closed under pullback.

\item ``Enhanced'' cocycles can be pulled back along pairs of weak equivalences.

\item The underlying cocycle of an ``enhanced'' cocycle is a $\mathcal{V}$-cocycle.

\item Every cocycle can be replaced with an ``enhanced'' cocycle in a homotopically unique way.
\end{itemize}
More precisely, we make the following definition.

\begin{dfn}
A \strong{homotopical calculus of cocycles} for $\mathcal{C}$ consists of a class $\mathcal{V}$ of morphisms in $\mathcal{W}$, a category $\FCorr[\mathcal{C}]$, and a functor $U : \FCorr[\mathcal{C}] \to \ZzCat{\bracket{-1 ; 1}}{\mathcal{C}}$ satisfying the following conditions:
\begin{itemize}
\item $\mathcal{V}$ is closed under pullback in $\mathcal{C}$ in the sense that, for any morphism $v : X \to Y$ in $\mathcal{V}$ and any morphism $g : Y' \to Y$ in $\mathcal{C}$, there is a pullback diagram in $\mathcal{C}$ of the form below,
\[
\begin{tikzcd}
X' \dar[swap]{v'} \rar &
X \dar{v} \\
Y' \rar[swap]{g} &
Y
\end{tikzcd}
\]
and in any such pullback diagram, $v' : X' \to Y'$ is also in $\mathcal{V}$.

\item Every isomorphism in $\mathcal{C}$ is a member of $\mathcal{V}$.

\item The composite
\[
\begin{tikzcd}[column sep=9.0ex]
\FCorr[\mathcal{C}] \rar{U} &
\ZzCat{\bracket{-1 ; 1}}{\mathcal{C}} \rar{\prodtuple{\dom, \codom}} &
\mathcal{W} \times \mathcal{W}
\end{tikzcd}
\]
is a Grothendieck fibration, where $\dom$ (\resp $\codom$) is the evident functor $\ZzCat{\bracket{-1 ; 1}}{\mathcal{C}} \to \mathcal{W}$ sending a cocycle $X \profto Y$ to $X$ (\resp $Y$), and the functor $U : \FCorr[\mathcal{C}] \to \ZzCat{\bracket{-1 ; 1}}{\mathcal{C}}$ preserves cartesian morphisms.

\item For each object $E$ in $\FCorr[\mathcal{C}]$, $U E$ is a $\mathcal{V}$-cocycle in $\mathcal{C}$.

\item For each pair $\tuple{X, Y}$ of objects in $\mathcal{C}$, writing $\FCorr[\mathcal{C}]{X}{Y}$ for the strict fibre of the above functor $\FCorr[\mathcal{C}] \to \mathcal{W} \times \mathcal{W}$, the induced functor 
\[
U_{X, Y} : \FCorr[\mathcal{C}]{X}{Y} \to \Cocyc[\mathcal{C}]{X}{Y}
\]
is homotopy cofinal.
\end{itemize}
\end{dfn}

\begin{remark}
We do \emph{not} require $\prodtuple{\dom, \codom} : \ZzCat{\bracket{-1 ; 1}}{\mathcal{C}} \to \mathcal{W} \times \mathcal{W}$ to be a Grothendieck fibration. Nonetheless, it still makes sense to talk about cartesian morphisms in $\ZzCat{\bracket{-1 ; 1}}{\mathcal{C}}$. For example, consider a cocycle in $\mathcal{C}$,
\[
\begin{tikzcd}
X \rar[leftarrow]{v} &
Z \rar{f} &
Y
\end{tikzcd}
\]
where $v : Z \to X$ is in $\mathcal{V}$; then, for any weak equivalence $w : X' \to X$ in $\mathcal{C}$, we can form the following commutative diagram in $\mathcal{C}$,
\[
\begin{tikzcd}
X' \dar[swap]{w} \rar[leftarrow]{v'} &
Z' \dar \rar{f'} &
Y \dar[equals] \\
X \rar[leftarrow, swap]{v} &
Z \rar[swap]{f} &
Y
\end{tikzcd}
\]
where the left square is a pullback diagram in $\mathcal{C}$, and it is straightforward to verify that the corresponding morphism $\tuple{f', v'} \to \tuple{f, v}$ is a cartesian morphism in $\ZzCat{\bracket{-1 ; 1}}{\mathcal{C}}$.
\end{remark}

The primary example of a homotopical calculus of cocycles is the case where $\mathcal{C}$ is a category of fibrant objects, $\mathcal{V}$ is the subcategory of trivial fibrations in $\mathcal{C}$, $\FCorr[\mathcal{C}]$ is a certain full subcategory of $\ZzCat{\bracket{-1 ; 1}}{\mathcal{C}}$, and the functor $U : \FCorr[\mathcal{C}] \to \ZzCat{\bracket{-1 ; 1}}{\mathcal{C}}$ is the inclusion. The details of this are deferred to the following sections.

\makenumpar
For the remainder of this section, let $\mathcal{V} \subseteq \mathcal{W}$, $\FCorr[\mathcal{C}]$, and $U : \FCorr[\mathcal{C}] \to \ZzCat{\bracket{-1 ; 1}}{\mathcal{C}}$ be the data of a homotopical calculus of cocycles for $\mathcal{C}$. 

\begin{lem}
Let $\mathcal{D}$ be a full subcategory of $\mathcal{C}$ (regarded as a relative category with the same weak equivalences) and let $U : \FCorr[\mathcal{D}] \to \ZzCat{\bracket{-1 ; 1}}{\mathcal{D}}$ be defined by the following pullback diagram in $\cat{\Cat}$:
\[
\begin{tikzcd}
\FCorr[\mathcal{D}] \dar[hookrightarrow] \rar{U} &
\ZzCat{\bracket{-1 ; 1}}{\mathcal{D}} \dar[hookrightarrow] \\
\FCorr[\mathcal{C}] \rar[swap]{U} &
\ZzCat{\bracket{-1 ; 1}}{\mathcal{C}}
\end{tikzcd}
\]
If $\mathcal{D}$ is a homotopically replete in $\mathcal{C}$,\footnote{--- \ie for any weak equivalence $w : X \to Y$ in $\mathcal{C}$, if either $X$ or $Y$ is in $\mathcal{D}$, then $X$, $Y$, and $w : X \to Y$ are all in $\mathcal{D}$.} then $\mathcal{V} \cap \mor \mathcal{D}$, $\FCorr[\mathcal{D}]$, and $U : \FCorr[\mathcal{D}] \to \ZzCat{\bracket{-1 ; 1}}{\mathcal{D}}$ define a homotopical calculus of cocycles in $\mathcal{D}$.
\end{lem}
\begin{proof}
Since $\mathcal{D}$ is a full and homotopically replete subcategory of $\mathcal{C}$, $\mathcal{V} \cap \mathcal{D}$ is closed under pullback in $\mathcal{D}$. It is not hard to verify that the following diagram is a pullback square in $\cat{\Cat}$,
\[
\begin{tikzcd}[column sep=9.0ex]
\ZzCat{\bracket{-1 ; 1}}{\mathcal{D}} \dar[hookrightarrow] \rar{\prodtuple{\dom, \codom}} &
\weq \mathcal{D} \times \weq \mathcal{D} \dar[hookrightarrow] \\
\ZzCat{\bracket{-1 ; 1}}{\mathcal{C}} \rar[swap]{\prodtuple{\dom, \codom}} &
\weq \mathcal{C} \times \weq \mathcal{C}
\end{tikzcd}
\]
so by the pullback pasting lemma, the outer rectangle in the diagram below is also a pullback diagram in $\cat{\Cat}$:
\[
\begin{tikzcd}[column sep=9.0ex]
\FCorr[\mathcal{D}] \dar[hookrightarrow] \rar{U} &
\ZzCat{\bracket{-1 ; 1}}{\mathcal{D}} \dar[hookrightarrow] \rar{\prodtuple{\dom, \codom}} &
\weq \mathcal{D} \times \weq \mathcal{D} \dar[hookrightarrow] \\
\FCorr[\mathcal{C}] \rar[swap]{U} &
\ZzCat{\bracket{-1 ; 1}}{\mathcal{C}} \rar[swap]{\prodtuple{\dom, \codom}} &
\weq \mathcal{C} \times \weq \mathcal{C}
\end{tikzcd}
\]
Recalling that the class of Grothendieck fibrations is closed under pullback in $\cat{\Cat}$, we deduce that the composite of the top row is a Grothendieck fibration, as required. Moreover, any morphism in $\ZzCat{\bracket{-1 ; 1}}{\mathcal{D}}$ that is a cartesian morphism in $\ZzCat{\bracket{-1 ; 1}}{\mathcal{C}}$ is automatically a cartesian morphism in $\ZzCat{\bracket{-1 ; 1}}{\mathcal{D}}$, so $U : \FCorr[\mathcal{D}] \to \ZzCat{\bracket{-1 ; 1}}{\mathcal{D}}$ preserves cartesian morphisms. Since the remaining axioms can be checked fibrewise, this completes the proof.
\end{proof}

\begin{lem}
\label{lem:replacing.cocycles.with.special.cocycles}
For any pair $\tuple{X, Y}$ of objects in $\mathcal{C}$, in the following commutative diagram,
\[
\begin{tikzcd}
\FCorr[\mathcal{C}]{X}{Y} \dar[equals] \rar &
\Cocyc[\mathcal{C}][\mathcal{V}]{X}{Y} \dar[hookrightarrow] \\
\FCorr[\mathcal{C}]{X}{Y} \rar[swap]{U_{X, Y}} &
\Cocyc[\mathcal{C}]{X}{Y}
\end{tikzcd}
\]
every arrow is a weak homotopy equivalence of categories.
\end{lem}
\begin{proof}
The bottom horizontal arrow is a homotopy cofinal functor and the right vertical arrow is fully faithful. By \autoref{lem:homotopy.cofinal.functors.and.fully.faithful.functors}, the top horizontal arrow is also a homotopy cofinal functor, so by Quillen's Theorem~A (\ref{thm:Quillen-A}) and the 2-out-of-3 property, the inclusion is indeed a weak homotopy equivalence of categories.
\end{proof}

\makenumpar
\label{par:replacing.weak.equivalences.with.special.correspondences}
Let $\mathcal{R}$ be the category defined as follows:
\begin{itemize}
\item The objects are tuples $\tuple{X, Y, w, E, u}$ where $w : X \to Y$ is a weak equivalence in $\mathcal{C}$, $E$ is an object in $\FCorr[\mathcal{C}]{X}{Y}$, and $u$ is a weak equivalence in $\mathcal{C}$ making the diagram in $\mathcal{C}$ shown below commute,
\[
\begin{tikzcd}
Y \dar[equals] \rar[leftarrow]{w} &
X \dar{u} \rar{\id} &
X \dar[equals] \\
Y \rar[leftarrow] &
\bullet \rar &
X
\end{tikzcd}
\]
where the bottom row is the cocycle $U E$. (In particular, we require $\parens{{\dom} \circ U} E = Y$ and $\parens{{\codom} \circ U} E = X$.)

\item The morphisms $\tuple{X_0, Y_0, w_0, E_0, u_0} \to \tuple{X_1, Y_1, w_1, E_1, u_1}$ are the morphisms $k : E_0 \to E_1$ in $\FCorr[\mathcal{C}]$ such that the diagram in $\mathcal{C}$ shown below commutes,
\[
\begin{tikzcd}[column sep=4.5ex, row sep=4.5ex]
Y_0 \arrow[equals]{dd} \drar \arrow[leftarrow]{rr}{w_0} &&
X_0 \arrow{dd}[near end]{u_0} \drar \arrow{rr}{\id} &&
X_0 \drar \arrow[equals]{dd} \\
&
Y_1 \arrow[leftarrow, crossing over]{rr}[near start]{w_1} &&
X_1 \arrow[crossing over]{rr}[near end]{\id} &&
X_1 \arrow[equals]{dd} \\
Y_0 \arrow[leftarrow]{rr} \drar &&
\bullet \arrow{rr} \drar &&
X_0 \drar \\
&
Y_1 \arrow[equals, crossing over]{uu} \arrow[leftarrow]{rr} &&
\bullet \arrow[leftarrow, crossing over]{uu}[swap, near start]{u_1} \arrow{rr} &&
X_1
\end{tikzcd}
\]
where the bottom face is $U k$.

\item Composition and identities are inherited from $\FCorr[\mathcal{C}]$.
\end{itemize}
Note that there is an evident fully faithful embedding of $\mathcal{R}$ into the comma category $\commacat{\ZzCat{\bracket{-1 ; 1}}{\mathcal{C}}}{U}$.

\begin{lem}
\label{lem:replacing.weak.equivalences.with.special.correspondences}
The evident projection $\mathcal{R} \to \ZzCat{\bracket{-1}}{\mathcal{C}}$ is a Grothendieck fibration whose (strict) fibres are weakly contractible.
\end{lem}
\begin{proof}
Let $\tuple{X_1, Y_1, w_1, E_1, u_1}$ be an object in $\mathcal{R}$.
Suppose we have the following commutative diagram in $\mathcal{W}$:
\[
\begin{tikzcd}
Y_0 \dar[swap]{g} \rar[leftarrow]{w_0} &
X_0 \dar{f} \\
Y_1 \rar[swap, leftarrow]{w_1} &
X_1
\end{tikzcd}
\]
Since $\FCorr[\mathcal{C}] \to \mathcal{W} \times \mathcal{W}$ is a Grothendieck fibration, we may choose a cartesian morphism $k : E_0 \to E_1$ in $\FCorr[\mathcal{C}]$ such that $U k$ is of the form below:
\[
\begin{tikzcd}
Y_0 \dar[swap]{g} \rar[leftarrow]{v_0} &
Z_0 \dar{h} \rar{q_0} &
X_0 \dar{f} \\
Y_1 \rar[swap, leftarrow]{v_1} &
Z_1 \rar[swap]{q_1} &
X_1
\end{tikzcd}
\]
Moreover, $U k : U E_0 \to U E_1$ is a cartesian morphism in $\ZzCat{\bracket{-1 ; 1}}{\mathcal{C}}$, so there is a unique morphism $u_0 : X_0 \to Z_0$ in $\mathcal{C}$ such that the diagram in $\mathcal{C}$ shown below commutes,
\[
\begin{tikzcd}[column sep=4.5ex, row sep=4.5ex]
Y_0 \arrow[equals]{dd} \drar{g} \arrow[leftarrow]{rr}{w_0} &&
X_0 \arrow[dashed]{dd}[near end]{u_0} \drar{f} \arrow{rr}{\id} &&
X_0 \drar{f} \arrow[equals]{dd} \\
&
Y_1 \arrow[leftarrow, crossing over]{rr}[near start]{w_1} &&
X_1 \arrow[crossing over]{rr}[near end]{\id} &&
X_1 \arrow[equals]{dd} \\
Y_0 \arrow[leftarrow]{rr}[swap, near start]{v_0} \drar[swap]{g} &&
Z_0 \arrow{rr}[swap, near end]{q_0} \drar[swap]{h} &&
X_0 \drar[swap]{f} \\
&
Y_1 \arrow[equals, crossing over]{uu} \arrow[leftarrow]{rr}[swap]{v_1} &&
Z_1 \arrow[leftarrow, crossing over]{uu}[swap, near start]{u_1} \arrow{rr}[swap]{q_1} &&
X_1
\end{tikzcd}
\]
and by the 2-out-of-3 property, $u_0 : X_0 \to Z_0$ is a weak equivalence in $\mathcal{C}$. Thus, $\tuple{X_0, Y_0, w_0, E_0, u_0}$ is an object in $\mathcal{R}$ and $k : E_0 \to E_1$ defines a morphism $\tuple{X_0, Y_0, w_0, E_0, u_0} \to \tuple{X_1, Y_1, w_1, E_1, u_1}$ in $\mathcal{R}$.

We will now show that $k : \tuple{X_0, Y_0, w_0, E_0, u_0} \to \tuple{X_1, Y_1, w_1, E_1, u_1}$ is a cartesian morphism in $\mathcal{R}$. Let $k' : \tuple{X', Y', w', E', u'} \to \tuple{X_1, Y_1, w_1, E_1, u_1}$ be a morphism in $\mathcal{R}$ and suppose we have a commutative diagram in $\mathcal{W}$ of the form below,
\[
\begin{tikzcd}
Y' \dar[swap]{y} \rar[leftarrow]{w'} &
X' \dar{x} \\
Y_0 \dar[swap]{g} \rar[leftarrow]{w_0} &
X_0 \dar{f} \\
Y_1 \rar[leftarrow, swap]{w_1} &
X_1
\end{tikzcd}
\]
where $\parens{{\dom} \circ U} k' = g \circ y$ and $\parens{{\codom} \circ U} k' = f \circ x$. Since $k : E_0 \to E_1$ is a cartesian morphism in $\FCorr[\mathcal{C}]$, there is a unique morphism $e : E' \to E_0$ such that $k \circ e = k'$ with $\parens{{\dom} \circ U} e = y$ and $\parens{{\codom} \circ U} e = x$. Suppose $U e : U E' \to U E_0$ is as follows:
\[
\begin{tikzcd}
Y' \dar[swap]{y} \rar[leftarrow]{v'} &
Z' \dar{z} \rar{q'} &
X' \dar{x} \\
Y_0 \rar[leftarrow, swap]{v_0} &
Z_0 \rar[swap]{q_0} &
X_0
\end{tikzcd}
\]
To show that $e$ defines a morphism $\tuple{X', Y', w', E', u'} \to \tuple{X_0, Y_0, w_0, E_0, u_0}$ in $\mathcal{R}$, we must verify that $z \circ u' = u_0 \circ x$. But the diagram in $\mathcal{C}$ shown below commutes,
\[
\begin{tikzcd}
Y' \dar[swap]{y} \rar[leftarrow]{w'} &
X' \dar{x} \rar{\id} &
X' \dar{x} \\
Y' \rar[leftarrow, swap]{v'} &
Z' \rar[swap]{q'} &
X'
\end{tikzcd}
\]
so $v_0 \circ \parens{z \circ u'} = v_0 \circ \parens{u_0 \circ x}$ and $q_0 \circ \parens{z \circ u'} = q_0 \circ \parens{u_0 \circ x}$. Furthermore, $h \circ \parens{z \circ u'} = \parens{u_1 \circ f} \circ x = h \circ \parens{u_0 \circ x}$, and since $U k : U E_0 \to U E_1$ is a cartesian morphism in $\ZzCat{\bracket{-1 ; 1}}{\mathcal{C}}$, it follows that $z \circ u' = u_0 \circ x$ as required. This completes the proof that $k : \tuple{X_0, Y_0, w_0, E_0, u_0} \to \tuple{X_1, Y_1, w_1, E_1, u_1}$ is a cartesian morphism in $\mathcal{R}$.

Finally, it remains to be shown that the (strict) fibres of $\mathcal{R} \to \ZzCat{\bracket{-1}}{\mathcal{C}}$ are weakly contractible. But for any weak equivalence $w : X \to Y$ in $\mathcal{C}$, the corresponding fibre is isomorphic to the comma category $\commacat{\tuple{\id_X, w}}{U_{Y, X}}$, and since $U_{Y, X} : \FCorr[\mathcal{C}]{Y}{X} \to \Cocyc[\mathcal{C}]{Y}{X}$ is a homotopy cofinal functor, $\commacat{\tuple{\id_X, w}}{U_{Y, X}}$ is weakly contractible, as required.
\end{proof}

\begin{lem}
\label{lem:reducing.extended.three-arrow.zigzags.to.extended.two-arrow.zigzags}
$\mathcal{C}$ admits a homotopical calculus of right fractions.
\end{lem}
\begin{proof}
Let $\tuple{X, Y}$ be a pair of objects in $\mathcal{C}$, let $k$ and $l$ be natural numbers, let $\mathcal{H}_0 \argp{X, Y} = \mathcal{C}^{\bracket{-1 ; k ; l}} \argp{X, Y}$, let $\mathcal{H}_1 \argp{X, Y} = \mathcal{C}^{\bracket{-1 ; k ; -1 ; l}}$, and let $S : \mathcal{H}_0 \argp{X, Y} \to \mathcal{H}_1 \argp{X, Y}$ be the evident functor defined by inserting an identity morphism. We must show that $S$ is a weak homotopy equivalence of categories.

Let $\mathcal{H}_1^+ \argp{X, Y}$ be defined by the following pullback diagram in $\cat{\Cat}$,
\[
\begin{tikzcd}
\mathcal{H}_1^+ \argp{X, Y} \dar[swap]{P_1} \rar{Q} &
\mathcal{R} \dar{P} \\
\mathcal{H}_1 \argp{X, Y} \rar &
\ZzCat{\bracket{-1}}{\mathcal{C}}
\end{tikzcd}
\]
where $\mathcal{H}_1 \argp{X, Y} \to \ZzCat{\bracket{-1}}{\mathcal{C}}$ the evident projection that sends a zigzag of type $\bracket{-1 ; k ; -1 ; l}$ to the interior leftward-pointing arrow and $P : \mathcal{R} \to \ZzCat{\bracket{-1}}{\mathcal{C}}$ is the Grothendieck fibration defined in \autoref{lem:replacing.weak.equivalences.with.special.correspondences}, and let $\mathcal{H}_0^+ \argp{X, Y}$ be defined by the following pullback diagram in $\cat{\Cat}$:
\[
\begin{tikzcd}
\mathcal{H}_0^+ \argp{X, Y} \dar[swap]{P_0} \rar{S^+} &
\mathcal{H}_1^+ \argp{X, Y} \dar{P_1} \\
\mathcal{H}_0 \argp{X, Y} \rar[swap]{S} &
\mathcal{H}_1 \argp{X, Y}
\end{tikzcd}
\]
We know $P : \mathcal{R} \to \ZzCat{\bracket{-1}}{\mathcal{C}}$ is a Grothendieck fibration with weakly contractible (strict) fibres, and these properties are preserved by pullback, so both $P_1 : \mathcal{H}_1^+ \argp{X, Y} \to \mathcal{H}_1 \argp{X, Y}$ and $P_0 : \mathcal{H}_0^+ \argp{X, Y} \to \mathcal{H}_0 \argp{X, Y}$ are also Grothendieck fibrations with weakly contractible (strict) fibres. Hence, by \autoref{lem:Grothendieck.fibrations.and.homotopy.cofinality} and Quillen's Theorem~A (\ref{thm:Quillen-A}), they are weak homotopy equivalences of categories. Thus, in view of the 2-out-of-6 property (of isomorphisms), to show that $S : \mathcal{H}_0 \argp{X, Y} \to \mathcal{H}_1 \argp{X, Y}$ is a weak homotopy equivalence of categories, it suffices to find a functor $D : \mathcal{H}_1^+ \argp{X, Y} \to \mathcal{H}_0 \argp{X, Y}$ such that the diagram in $\Ho \cat{\SSet}$ shown below commutes:
\[
\tag{$\ast$}
\begin{tikzcd}
\nv{\mathcal{H}_0^+ \argp{X, Y}} \drar[swap, bend right=15]{\nv{P_0}} \rar{\nv{S^+}} &
\nv{\mathcal{H}_1^+ \argp{X, Y}} \dar{\nv{D}} \drar[bend left=15]{\nv{P_1}} \\
&
\nv{\mathcal{H}_0 \argp{X, Y}} \rar[swap]{\nv{S}} &
\nv{\mathcal{H}_1 \argp{X, Y}}
\end{tikzcd}
\]

First, observe that every object in $\mathcal{H}_1^+ \argp{X, Y}$ has an underlying commutative diagram in $\mathcal{C}$ of the form below:
\[
\hspace{-1in}
\begin{tikzcd}
{} &&&&
\tilde{Y}_0 \dar[leftarrow, swap]{q} \rar[leftarrow]{\id} &
\tilde{Y}_0 \dar[equals] \\
&&&&
\tilde{X}'_k \dar[swap]{v_k} \rar[leftarrow]{u_k} &
\tilde{Y}_0 \dar[equals] \\
X \rar[leftarrow] &
\tilde{X}_0 \rar &
\cdots \rar &
\tilde{X}_{k-1} \rar[swap]{f_k} &
\tilde{X}_k \rar[leftarrow, swap]{w} &
\tilde{Y}_{0} \rar &
\cdots \rar &
\tilde{Y}_{l-1} \rar[swap]{g_l} &
Y
\end{tikzcd}
\hspace{-1in}
\]
For $0 < i \le k$, write $f_i$ for the morphism $\tilde{X}_{i-1} \to \tilde{X}_i$ in the above diagram. Since $v_k : \tilde{X}'_k \to \tilde{X}_k$ is in $\mathcal{V}$, we may functorially construct the following commutative diagram in $\mathcal{C}$,
\[
\begin{tikzcd}
\tilde{X}'_0 \dar[swap]{v_0} \rar &
\cdots \rar &
\tilde{X}'_{k-1} \dar[swap]{v_{k-1}} \rar{f'_k} &
\tilde{X}'_k \dar{v_k} \\
\tilde{X}_0 \rar &
\cdots \rar &
\tilde{X}_{k-1} \rar[swap]{f_k} &
\tilde{X}_k
\end{tikzcd}
\]
where each square is a pullback diagram in $\mathcal{C}$. We then obtain the diagram in $\mathcal{C}$ shown below,
\[
\hspace{-1in}
\begin{tikzcd}
X \dar[equals] \rar[leftarrow] &
\tilde{X}'_0 \dar[equals] \rar &
\cdots \rar &
\tilde{X}'_{k-1} \dar[equals] \rar{q \circ f'_k} &
\tilde{Y}_0 \dar[leftarrow]{q} \rar[leftarrow]{\id} &
\tilde{Y}_0 \dar[equals] \rar &
\cdots \rar &
\tilde{Y}_{l-1} \dar[equals] \rar{g_l} &
\tilde{Y}_l \dar[equals] \\
X \dar[equals] \rar[leftarrow] &
\tilde{X}'_0 \dar[swap]{v_0} \rar &
\cdots \rar &
\tilde{X}'_{k-1} \dar[swap]{v_{k-1}} \rar{f'_k} &
\tilde{X}'_k \dar{v_k} \rar[leftarrow]{u_k} &
\tilde{Y}_0 \dar[equals] \rar &
\cdots \rar &
\tilde{Y}_{l-1} \dar[equals] \rar{g_l} &
\tilde{Y}_l \dar[equals] \\
X \rar[leftarrow] &
\tilde{X}_0 \rar &
\cdots \rar &
\tilde{X}_{k-1} \rar[swap]{f_k} &
\tilde{X}_k \rar[leftarrow, swap]{w} &
\tilde{Y}_0 \rar &
\cdots \rar &
\tilde{Y}_{l-1} \rar[swap]{g_l} &
\tilde{Y}_l
\end{tikzcd}
\hspace{-1in}
\]
where every vertical arrow is a weak equivalence in $\mathcal{C}$. Omitting the internal leftward-pointing arrow in the top row gives an object in $\mathcal{H}_0 \argp{X, Y}$, so this construction defines a functor $D : \mathcal{H}_1^+ \argp{X, Y} \to \mathcal{H}_0 \argp{X, Y}$ equipped with a zigzag of natural weak equivalences connecting $P_1$ and $S \circ D$.

Now, suppose $\tilde{X}_k = \tilde{Y}_0$ and $w = \id_{\tilde{Y}_0}$. Then, for $0 \le i < k$, there is a unique morphism $u_i : \tilde{X}_i \to \tilde{X}'_i$ in $\mathcal{C}$ making the diagram below commute:
\[
\begin{tikzcd}
\tilde{X}_i \arrow[bend right=30]{dd}[swap]{\id} \dar[dashed]{u_i} \rar{f_{i+1}} &
\tilde{X}_{i+1} \dar{u_{i+1}} \\
\tilde{X}'_i \dar{v_i} \rar{f'_{i+1}} &
\tilde{X}'_{i+1} \dar{v_{i+1}} \\
\tilde{X}_i \rar[swap]{f_{i+1}} &
\tilde{X}_{i+1}
\end{tikzcd}
\]
Since $q \circ u_k = \id_{\tilde{Y}_0}$, we obtain the following commutative diagram in $\mathcal{C}$,
\[
\hspace{-1in}
\begin{tikzcd}
X \dar[equals] \rar[leftarrow] & 
\tilde{X}_0 \dar[swap]{u_0} \rar &
\cdots \rar &
\tilde{X}_{k-1} \dar[swap]{u_{k-1}} \rar{f_k} &
\tilde{Y}_0 \dar[equals] \rar &
\cdots \rar &
\tilde{Y}_{l-1} \dar[equals] \rar{g_l} &
\tilde{Y}_l \dar[equals] \\
X \rar[leftarrow] & 
\tilde{X}'_0 \rar &
\cdots \rar &
\tilde{X}'_{k-1} \rar[swap]{q \circ f'_k} &
\tilde{Y}_0 \rar &
\cdots \rar &
\tilde{Y}_{l-1} \rar[swap]{g_l} &
\tilde{Y}_l
\end{tikzcd}
\hspace{-1in}
\]
where every vertical arrow is weak equivalence in $\mathcal{C}$. Thus, we have a natural weak equivalence $P_0 \hoto D \circ S^+$, so $D : \mathcal{H}_1^+ \argp{X, Y} \to \mathcal{H}_0 \argp{X, Y}$ is indeed a functor making ($\ast$) commute.
\end{proof}

A homotopical calculus of cocycles gives us a slightly better model for the homotopy type of the hom-spaces of the hammock localisation than a homotopical calculus of right fractions. Indeed, morphisms in $\mathcal{V}$ can be pulled back along arbitrary morphisms in $\mathcal{C}$, so the $\mathcal{V}$-cocycle category $\Cocyc[\mathcal{C}][\mathcal{V}]{X}{Y}$ is contravariantly pseudofunctorial in $X$ and strictly functorial in $Y$, as $X$ and $Y$ vary in $\mathcal{C}$ (and not just $\mathcal{W}$). We then have the following analogue of \autoref{thm:fundamental.theorem.of.homotopical.two-arrow.calculi} for $\mathcal{V}$-cocycles:

\begin{thm}
\label{thm:fundamental.theorem.of.homotopical.calculi.of.cocycles}
Let $\ul{\LH \mathcal{C}}$ be the hammock localisation of $\mathcal{C}$ and let $X$ and $Y$ be objects in $\mathcal{C}$.
\begin{enumerate}[(i)]
\item The reduction morphism $\nv{\Cocyc[\mathcal{C}][\mathcal{V}]{X}{Y}} \to \ulHom[\LH \mathcal{C}]{X}{Y}$ is a weak homotopy equivalence.

\item The reduction morphism $\nv{\Cocyc[\mathcal{C}][\mathcal{V}]{X}{Y}} \to \ulHom[\LH \mathcal{C}]{X}{Y}$ is natural in the following sense: for any morphisms $X' \to X$ and $Y \to Y'$ in $\mathcal{C}$, the following diagram commutes in $\Ho \cat{\SSet}$,
\[
\begin{tikzcd}
\nv{\Cocyc[\mathcal{C}][\mathcal{V}]{X}{Y}} \dar \rar &
\ulHom[\LH \mathcal{C}]{X}{X} \dar \\
\nv{\Cocyc[\mathcal{C}][\mathcal{V}]{X'}{Y'}} \rar &
\ulHom[\LH \mathcal{C}]{X'}{Y'}
\end{tikzcd}
\]
where the left vertical arrow is defined as above and the right vertical arrow is defined by concatenation.

\item There is an isomorphism
\[
\nv{\Cocyc[\mathcal{C}][\mathcal{V}]{\blank}{\blank}} \cong \ulHom[\LH \mathcal{C}]{\blank}{\blank}
\]
of functors $\op{\Ho \mathcal{C}} \times \Ho \mathcal{C} \to \Ho \cat{\SSet}$.
\end{enumerate}
\end{thm}
\begin{proof}
(i). Combine \autoref{thm:fundamental.theorem.of.homotopical.two-arrow.calculi} with  lemmas~\ref{lem:replacing.cocycles.with.special.cocycles} and~\ref{lem:reducing.extended.three-arrow.zigzags.to.extended.two-arrow.zigzags}.

\bigskip\noindent
(ii). Straightforward.

\bigskip\noindent
(iii). This is an immediate consequence of (i) and (ii).
\end{proof}

\begin{remark}
Moreover, if $\mathcal{V}$ is closed under composition, then we can form a bicategory whose hom-categories are the cocycle categories $\Cocyc[\mathcal{C}][\mathcal{V}]{X}{Y}$, which we may regard as a ``more algebraic'' model of $\ul{\LH \mathcal{C}}$.
\end{remark}

\section{Categories of fibrant objects}

The following definition is due to \citet{Brown:1973}.

\begin{dfn}
A \strong{category of fibrant objects} is a category $\mathcal{C}$ with finite products and equipped with a pair $\tuple{\mathcal{W}, \mathcal{F}}$ of subclasses of $\mor \mathcal{C}$ satisfying these axioms:
\begin{enumerate}[(A)]
\item $\tuple{\mathcal{C}, \mathcal{W}}$ is a category with weak equivalences.

\item Every isomorphism is in $\mathcal{F}$, and $\mathcal{F}$ is closed under composition.

\item Pullbacks along morphisms in $\mathcal{F}$ exist in $\mathcal{C}$, and the pullback of a morphism that is in $\mathcal{F}$ (\resp $\mathcal{W} \cap \mathcal{F}$) is also a morphism that is in $\mathcal{F}$ (\resp $\mathcal{W} \cap \mathcal{F}$).

\item For each object $X$ in $\mathcal{C}$, there is a commutative diagram of the form below,
\[
\begin{tikzcd}
X \drar[swap]{\Delta} \rar{i} &
\Path{X} \dar{p} \\
&
X \times X
\end{tikzcd}
\]
where $\Delta : X \to X \times X$ is the diagonal morphism, $i : X \to \Path{X}$ is in $\mathcal{W}$, and $\Path{X} \to X \times X$ is in $\mathcal{F}$.

\item For any object $X$ in $\mathcal{C}$, the unique morphism $X \to 1$ is in $\mathcal{F}$.
\end{enumerate}
\needspace{2.5\baselineskip}
In a category of fibrant objects as above,
\begin{itemize}
\item a \strong{weak equivalence} is a morphism in $\mathcal{W}$,

\item a \strong{fibration} is a morphism in $\mathcal{F}$, and

\item a \strong{trivial fibration} (or \strong{acyclic fibration}) is a morphism in $\mathcal{W} \cap \mathcal{F}$.
\end{itemize}
\end{dfn}

\begin{example}
Of course, the full subcategory of fibrant objects in a model category is a category of fibrant objects, with weak equivalences and fibrations inherited from the model structure.
\end{example}

\begin{example}
Let $\mathcal{M}$ be a right proper model category, let $\mathcal{W}$ be the class of weak equivalences, and let $\mathcal{F}$ be the class of morphisms $p : X \to Y$ with the following property: every pullback square in $\mathcal{M}$ of the form below
\[
\begin{tikzcd}
X' \dar \rar &
X \dar{p} \\
Y' \rar &
Y
\end{tikzcd}
\]
is also a homotopy pullback square in $\mathcal{M}$. (In other words, $\mathcal{F}$ is the class of \strong{sharp maps} in $\mathcal{M}$ in the sense of \citet{Rezk:1998}.) Let $\mathcal{E}$ be the full subcategory of $\mathcal{M}$ spanned by those objects $X$ such that the unique morphism $X \to 1$ is in $\mathcal{F}$. Then $\mathcal{E}$ is a category of fibrant objects, with weak equivalences $\mathcal{W}$ and fibrations $\mathcal{F}$.
\end{example}

\begin{dfn}
Let $\mathcal{C}$ be a category of fibrant objects and let $\tuple{X, Y}$ be a pair of objects in $\mathcal{C}$. A \strong{functional correspondence} $\tuple{p, v} : X \profto Y$ is a cocycle,
\[
\begin{tikzcd}
X \rar[leftarrow]{v} &
\tilde{X} \rar{p} &
Y
\end{tikzcd}
\]
such that $\prodtuple{p, v} : \tilde{X} \to Y \times X$ is a fibration.

We write $\FCorr[\mathcal{C}]$ (\resp $\FCorr[\mathcal{C}]{X}{Y}$) for the full subcategory of $\ZzCat{\bracket{-1 ; 1}}{\mathcal{C}}$ (\resp $\Cocyc[\mathcal{C}]{X}{Y}$) spanned by the functional correspondences.
\end{dfn}

\begin{remark}
Since product projections in a category of fibrant objects are fibrations, it follows that $v : \tilde{X} \to X$ is a trivial fibration and $p : \tilde{X} \to Y$ is a fibration. However, the converse is not true: for instance, $\tuple{\id_Y, \id_Y} : Y \profto Y$ is rarely a functional correspondence.
\end{remark}

\begin{lem}
\label{lem:Grothendieck.fibration.of.functional.correspondences}
Let $\mathcal{C}$ be a category of fibrant objects and let $\mathcal{W} = \weq \mathcal{C}$.
\begin{enumerate}[(i)]
\item Consider a pullback diagram in $\mathcal{C}$ of the form below,
\[
\begin{tikzcd}
\tilde{X}' \dar[swap]{\prodtuple{p', v'}} \rar{\tilde{f}} &
\tilde{X} \dar{\prodtuple{p, v}} \\
Y' \times X' \rar[swap]{g \times f} &
Y \times X
\end{tikzcd}
\]
where $f : X' \to X$ and $g : Y' \to Y$ are weak equivalences in $\mathcal{C}$. If $\tuple{p, v} : X \profto Y$ is a functional correspondence in $\mathcal{C}$, then $\tuple{p', v'} : X' \profto Y'$ is also a functional correspondence in $\mathcal{C}$.

\item The functor $\FCorr[\mathcal{C}] \to \mathcal{W} \times \mathcal{W}$ sending functional correspondences $X \profto Y$ to the pair $\tuple{X, Y}$ is a Grothendieck fibration.

\item The inclusion $\FCorr[\mathcal{C}] \embedinto \ZzCat{\bracket{-1 ; 1}}{\mathcal{C}}$ preserves cartesian morphisms.
\end{enumerate}
\end{lem}
\begin{proof}
(i). In view of the pullback pasting lemma, we may assume that either $f = \id_X$ or $g = \id_Y$.

First, consider the case where $g = \id_Y$. We have the following commutative diagram in $\mathcal{C}$,
\[
\begin{tikzcd}
\tilde{X}' \dar[swap]{\prodtuple{p', v'}} \rar{\tilde{f}} &
\tilde{X} \dar{\prodtuple{p, v}} \\
Y \times X' \dar \rar[swap]{\id_Y \times f} &
Y \times X \dar \rar &
Y \dar \\
X' \rar[swap]{f} &
X \rar &
1
\end{tikzcd}
\]
where the unlabelled arrows are the evident projections and the squares are pullback diagrams in $\mathcal{C}$. Thus, $v' : \tilde{X}' \to X'$ is the pullback of $v : \tilde{X} \to X$ along $f : X' \to X$, so by axiom~C, $v' : \tilde{X}' \to X'$ is indeed a trivial fibration in $\mathcal{C}$.

Now, consider the case where $f = \id_X$ instead. We have the following commutative diagram in $\mathcal{C}$,
\[
\begin{tikzcd}
\tilde{X}' \dar[swap]{\prodtuple{p', v'}} \rar{\tilde{f}} &
\tilde{X} \dar{\prodtuple{p, v}} \\
Y' \times X \dar \rar[swap]{g \times \id_X} &
Y \times X \dar \rar &
X \dar \\
Y' \rar[swap]{g} &
Y \rar &
1
\end{tikzcd}
\]
where the unlabelled arrows are the evident projections and the squares are pullback diagrams in $\mathcal{C}$. It is well known that the class of weak equivalences in $\mathcal{C}$ is closed under pullback along fibrations,\footnote{See \eg Lemma~2 in \citep[\Sect 4]{Brown:1973} or Lemma~8.5 in \citep[\Chap II]{GJ}.} so $\tilde{f} : \tilde{X}' \to \tilde{X}$ is a weak equivalence in $\mathcal{C}$. But $v' = v \circ \tilde{f}$, hence (by axiom~A) $v' : \tilde{X}' \to X$ is a trivial fibration in $\mathcal{C}$, as required.

\bigskip\noindent
(ii) and (iii). With notation as in (i), it is straightforward to verify that the commutative diagram
\[
\begin{tikzcd}
X' \dar[swap]{f} \rar[leftarrow]{v'} &
\tilde{X}' \dar{\tilde{f}} \rar{p'} &
Y' \dar{g} \\
X \rar[leftarrow, swap]{v} &
\tilde{X} \rar[swap]{p} &
Y
\end{tikzcd}
\]
defines a cartesian morphism in both $\FCorr[\mathcal{C}]$ and $\ZzCat{\bracket{-1 ; 1}}{\mathcal{C}}$.
\end{proof}

\needspace{3\baselineskip}
It is convenient to slightly strengthen the axioms given earlier.

\begin{dfn}
A \strong{path object functor} for a category of fibrant objects $\mathcal{C}$ consists of the following data:
\begin{itemize}
\item A functor $\Path : \mathcal{C} \to \mathcal{C}$.

\item Natural transformations $i : \id_{\mathcal{C}} \hoto \Path$ and $p_0, p_1 : \Path \hoto \id_{\mathcal{E}}$ such that $\tuple{\Path{C}, i_X, \parens{p_0}_X, \parens{p_1}_X}$ is a path object for every object $X$ in $\mathcal{C}$, \ie $i_X : X \to \Path{X}$ is a weak equivalence, $\prodtuple{\parens{p_0}_X, \parens{p_1}_X} : \Path{X} \to X \times X$ is a fibration and $\parens{p_0}_X \circ i_X = \parens{p_1}_X \circ i_X = \id_X$.
\end{itemize}
We say $\mathcal{C}$ has \strong{functorial path objects} if it admits a path object functor.
\end{dfn}

\begin{lem}[Factorisation lemma]
\label{lem:Ken.Brown.factorisation}
\needspace{3\baselineskip}
Let $f : X \to Y$ be a morphism in a category of fibrant objects $\mathcal{C}$.
\begin{enumerate}[(i)]
\item There exists a commutative diagram in $\mathcal{C}$ of the form below,
\[
\begin{tikzcd}
X \dar[equals] \rar[leftarrow]{\id} &
X \dar[swap]{u} \rar{f} &
Y \dar[equals] \\
X \rar[leftarrow, swap]{v} &
E_f \rar[swap]{p} &
Y
\end{tikzcd}
\]
where the bottom row is a functional correspondence in $\mathcal{C}$.

\item Moreover, if $\mathcal{C}$ has functorial path objects, then $u$, $v$, and $p$ can be chosen functorially (with respect to $f$).
\end{enumerate}
\end{lem}
\begin{proof} \openproof
See (the proof of) the factorisation lemma in \citep{Brown:1973}.
\end{proof}

\begin{lem}
\label{lem:homotopical.uniqueness.of.correspondence.replacements.via.functoriality}
Let $\mathcal{C}$ be a category of fibrant objects and let $\tuple{X, Y}$ be a pair of objects in $\mathcal{C}$. If $\mathcal{C}$ has functorial path objects, then the inclusion $U_{X, Y} : \FCorr[\mathcal{C}]{X}{Y} \embedinto \Cocyc[\mathcal{C}]{X}{Y}$ is homotopy cofinal.
\end{lem}
\begin{proof}
Let $\tuple{f, w} : X \profto Y$ be a cocycle in $\mathcal{C}$. We must show that the comma category $\commacat{\tuple{f, w}}{U_{X, Y}}$ is weakly contractible. By \autoref{lem:Ken.Brown.factorisation}, we may factor $\prodtuple{f, w} : \tilde{X} \to Y \times X$ as a weak equivalence followed by a fibration, yielding an object in $\commacat{\tuple{f, w}}{U_{X, Y}}$. We may then use the functoriality of this factorisation to construct a zigzag of natural weak equivalences between $\id_{\commacat{\tuple{f, w}}{U_{X, Y}}}$ and a constant endofunctor, and it follows that $\commacat{\tuple{f, w}}{U_{X, Y}}$ is weakly contractible.
\end{proof}

\begin{remark}
The argument in the proof above is essentially the same as the proof of Theorem 14.6.2 in \citep{Hirschhorn:2003}, but applied in a different context.
\end{remark}

\begin{thm}
\label{thm:categories.of.fibrant.objects.with.functorial.path.objects.admit.a.homotopical.calculus.of.cocycles}
Let $\mathcal{C}$ be a category of fibrant objects and let $\mathcal{V}$ be the subcategory of trivial fibrations in $\mathcal{C}$. If $\mathcal{C}$ has functorial path objects, then $\mathcal{V} \subseteq \mor \weq \mathcal{C}$, $\FCorr[\mathcal{C}]$, and $\FCorr[\mathcal{C}] \embedinto \ZzCat{\bracket{-1 ; 1}}{\mathcal{C}}$ constitute a homotopical calculus of cocycles in $\mathcal{C}$.
\end{thm}
\begin{proof}
Combine lemmas~\ref{lem:Grothendieck.fibration.of.functional.correspondences} and~\ref{lem:homotopical.uniqueness.of.correspondence.replacements.via.functoriality}.
\end{proof}

To extend the above result to the case where $\mathcal{C}$ is not assumed to have functorial path objects, we would have to prove \autoref{lem:homotopical.uniqueness.of.correspondence.replacements.via.functoriality} without using functorial factorisations. We will do this in the appendix.

\section{Simplicial categories of fibrant objects}

One way of getting a category of fibrant objects with functorial path objects is to take the full subcategory of fibrant objects in a simplicial closed model category. We may treat these axiomatically as follows:

\begin{dfn}
A \strong{simplicial category of fibrant objects} is a simplicially enriched category $\ul{\mathcal{C}}$ with \emph{simplicially enriched} finite products and equipped with a pair $\tuple{\mathcal{W}, \mathcal{F}}$ of subclasses of $\mor \mathcal{C}$ satisfying axioms A, B, C, E, and these additional axioms:
\begin{itemize}
\item[(C$\sb{\Delta}$)] \emph{Simplicially enriched} pullbacks along morphisms in $\mathcal{F}$ exist in $\mathcal{C}$.

\item[(F)] For any finite simplicial set $K$ and any object $X$ in $\mathcal{C}$, there exists an object $K \cotens X$ in $\mathcal{C}$ equipped with a (simplicially enriched natural) isomorphism
\[
\ulHom[\SSet]{K}{\ulHom[\mathcal{C}]{\blank}{X}} \cong \ulHom[\mathcal{C}]{\blank}{K \cotens X}
\]
of simplicially enriched functors $\op{\ul{\mathcal{C}}} \to \ul{\cat{\SSet}}$.

\item[(G)] For any monomorphism $i : K \to L$ of finite simplicial sets and any fibration $p : X \to Y$ in $\mathcal{C}$, the morphisms
\begin{align*}
i \cotens \id_Y & : L \cotens Y \to K \cotens Y &
\id_K \cotens p & : K \cotens X \to K \cotens Y
\end{align*}
are fibrations in $\mathcal{C}$, and the morphism
\[
i \pbprod p : L \cotens X \to \parens{K \cotens X} \times_{K \cotens Y} \parens{L \cotens Y}
\]
induced by the commutative diagram in $\mathcal{C}$ shown below
\[
\begin{tikzcd}
L \cotens X \dar[swap]{i \cotens \id_X} \rar{\id_L \cotens p} &
L \cotens Y \dar{i \cotens \id_Y} \\
K \cotens X \rar[swap]{\id_K \cotens p} &
K \cotens Y
\end{tikzcd}
\]
is a fibration in $\mathcal{C}$. Moreover, if $i$ is an anodyne extension (\resp $p$ is a trivial fibration), then both $i \cotens \id_Y$ (\resp $\id_K \cotens p$) and $i \pbprod p$ are trivial fibrations.
\end{itemize}
\end{dfn}

\begin{example}
Of course, if $\ul{\mathcal{M}}$ is a simplicial closed model category and $\ul{\mathcal{M}}_\mathrm{f}$ is the simplicially enriched full subcategory of fibrant objects, then $\ul{\mathcal{M}}_\mathrm{f}$ admits the structure of a simplicial category of fibrant objects with the weak equivalences and fibrations inherited from $\mathcal{M}$.
\end{example}

\begin{prop}
\label{prop:simplicial.categories.of.fibrant.objects.have.functorial.path.objects}
Let $\ul{\mathcal{C}}$ be a simplicial category of fibrant objects. Then the underlying ordinary category $\mathcal{C}$ (satisfies axiom D and) is a category of fibrant objects with functorial path objects.
\end{prop}
\begin{proof}
It straightforward to verify that $\Delta^1 \cotens \pblank$ is (the functor part of) a path object functor for $\mathcal{C}$.
\end{proof}

\begin{lem}
\label{lem:category.of.trivial.fibrations.is.cotensored}
Let $\ul{\mathcal{C}}$ be a simplicial category of fibrant objects, let $X$ be an object in $\mathcal{C}$, let $\ul{\mathcal{Q}}$ be the full subcategory of the simplicially enriched slice category $\overcat{\ul{\mathcal{C}}}{X}$ spanned by the trivial fibrations, and let $p : U \to X$ be an object in $\mathcal{Q}$, \ie a trivial fibration in $\mathcal{C}$.
\begin{enumerate}[(i)]
\item For any finite simplicial set $K$, the cotensor product $K \cotens_X p : K \cotens_X U \to X$ exists in $\ul{\mathcal{Q}}$.

\item For any monomorphism $i : K \to L$ of finite simplicial sets, the induced morphism $i \cotens_X U : L \cotens_X U \to K \cotens_X U$ is a trivial fibration in $\mathcal{C}$.
\end{enumerate}
\end{lem}
\begin{proof}
(i). Define the object $K \cotens_X p : K \cotens_X U \to X$ in $\overcat{\mathcal{C}}{X}$ by the following pullback diagram in $\mathcal{C}$,
\[
\begin{tikzcd}
K \cotens_X U \dar[swap]{K \cotens_X p} \rar &
K \cotens U \dar{\id_K \cotens p} \\
X \rar &
K \cotens X
\end{tikzcd}
\]
where the bottom arrow is the morphism induced by the unique morphism $K \to \Delta^0$. By axiom G, $\id_K \cotens p : K \cotens U \to K \cotens X$ is a trivial fibration in $\mathcal{C}$, so by axiom C, $K \cotens_X p : K \cotens_X U \to X$ is also a trivial fibration in $\mathcal{C}$, hence is an object in $\mathcal{Q}$. It is straightforward to verify that $K \cotens_X p$ has the required simplicially enriched universal property in $\ul{\mathcal{Q}}$.

\bigskip\noindent
(ii). By axiom G, we have a trivial fibration
\[
i \pbprod p : L \cotens U \to \parens{K \cotens U} \times_{K \cotens X} \parens{L \cotens X}
\]
induced by the commutative diagram in $\mathcal{C}$ shown below:
\[
\begin{tikzcd}
L \cotens U \dar[swap]{i \cotens \id_U} \rar{\id_L \cotens p} &
L \cotens X \dar{i \cotens \id_X} \\
K \cotens U \rar[swap]{\id_K \cotens p} &
K \cotens X
\end{tikzcd}
\]
Moreover, by the pullback pasting lemma, we have the following commutative diagram in $\mathcal{C}$,
\[
\begin{tikzcd}
L \cotens_X U \dar[swap]{i \cotens_X U} \rar &
L \cotens U \dar{i \pbprod p} \drar{i \cotens \id_U} \\
K \cotens_X U \dar[swap]{Z \cotens_X p} \rar &
\parens{K \cotens U} \times_{K \cotens X} \parens{L \cotens X} \dar \rar &
K \cotens U \dar{\id_K \cotens p} \\
X \rar &
L \cotens X \rar[swap]{i \cotens \id_X} &
K \cotens X
\end{tikzcd}
\]
where every square and rectangle is a pullback diagram in $\mathcal{C}$. Thus, by axiom C, $i \cotens_X U : L \cotens_X U \to K \cotens_X U$ is indeed a trivial fibration in $\mathcal{C}$.
\end{proof}

\begin{lem}
\label{lem:category.of.trivial.fibrations.is.homotopy.cofiltered}
%Let $\ul{\mathcal{C}}$ be a simplicial category of fibrant objects, let $X$ be an object in $\mathcal{C}$, let $\ul{\mathcal{Q}}$ be the full subcategory of the simplicially enriched slice category $\overcat{\ul{\mathcal{C}}}{X}$ spanned by the trivial fibrations.
With notation as in \autoref{lem:category.of.trivial.fibrations.is.cotensored}:
\begin{enumerate}[(i)]
\item $\ul{\mathcal{Q}}$ has simplicially enriched finite products.

\item Given any monomorphism $i : K \to L$ of finite simplicial sets and any pair $\tuple{p', p}$ of objects in $\mathcal{Q}$, for each morphism $f : K \to \ulHom[\mathcal{Q}]{p'}{p}$, there exist a morphism $v : p'' \to p'$ in $\mathcal{Q}$ and a morphism $g : L \to \ulHom[\mathcal{Q}]{p''}{p}$ making the following diagram commute:
\[
\begin{tikzcd}
K \dar[swap]{i} \rar{f} &
\ulHom[\mathcal{Q}]{p'}{p} \dar{\ulHom[\mathcal{Q}]{v}{p}} \\
L \rar[swap]{g} &
\ulHom[\mathcal{Q}]{p''}{p}
\end{tikzcd}
\]
\end{enumerate}
\end{lem}
\begin{proof}
(i). It is clear that $\ul{\mathcal{Q}}$ has a simplicially enriched terminal object, and the existence of simplicially enriched binary products is an immediate consequence of axioms C and C$\sb{\Delta}$.

\bigskip\noindent
(ii). By \autoref{lem:category.of.trivial.fibrations.is.cotensored}, $f : K \to \ulHom[\mathcal{Q}]{p'}{p}$ corresponds to a morphism $\tilde{f} : p' \to K \cotens_X p$ in $\mathcal{Q}$, and by axiom C, we may form the following pullback diagram in $\mathcal{Q}$,
\[
\begin{tikzcd}
p'' \dar[swap]{v} \rar{\tilde{g}} &
L \cotens_X p \dar{i \cotens_X p} \\
p' \rar[swap]{\tilde{f}} &
K \cotens_X p
\end{tikzcd}
\]
where (the underlying morphism of) $v : p'' \to p'$ is a trivial fibration in $\mathcal{C}$. Then $\tilde{g} : p'' \to L \cotens_X p$ corresponds to a morphism $g : L \to \ulHom[\mathcal{Q}]{p''}{p}$, and it is straightforward to see that diagram in question commutes.
\end{proof}

\begin{cor}
Let $\ul{\mathcal{C}}$ be a simplicial category of fibrant objects, let $X$ be an object in $\mathcal{C}$, let $\ul{\mathcal{Q}}$ be the full subcategory of the simplicially enriched slice category $\overcat{\ul{\mathcal{C}}}{X}$ spanned by the trivial fibrations, and let $\pi_0 \argb{\ul{\mathcal{Q}}}$ be the category obtained by applying $\pi_0$ to the hom-spaces of $\ul{\mathcal{Q}}$. Then $\op{\pi_0 \argb{\ul{\mathcal{Q}}}}$ is a filtered category.
\end{cor}
\begin{proof}
Recalling \autoref{lem:category.of.trivial.fibrations.is.homotopy.cofiltered}, it suffices to show that, for any parallel pair $f_0, f_1 : p' \to p$ in $\mathcal{Q}$, there is a morphism $v : p'' \to p'$ in $\mathcal{Q}$ such that $f_0 \circ v = f_1 \circ v$ in $\pi_0 \argb{\ul{\mathcal{Q}}}$. But $\tuple{f_0, f_1}$ define a morphism $f : \partial \Delta^1 \to \ulHom[\mathcal{Q}]{p'}{p}$, so the lemma implies there exist a morphism $v : p'' \to p'$ in $\mathcal{Q}$ and a morphism $g : \Delta^1 \to \ulHom[\mathcal{Q}]{p''}{p}$ making the diagram below commute,
\[
\begin{tikzcd}
\partial \Delta^1 \dar[hookrightarrow] \rar{f} &
\ulHom[\mathcal{Q}]{p'}{p} \dar{\ulHom[\mathcal{Q}]{v}{p}} \\
\Delta^1 \rar[swap]{g} &
\ulHom[\mathcal{Q}]{p''}{p}
\end{tikzcd}
\]
so we indeed have $f_0 \circ v = f_1 \circ v$ in $\pi_0 \ulHom[\mathcal{Q}]{p''}{p}$.
\end{proof}

The next result may be regarded as a homotopical version of Theorem 1 in \citep{Brown:1973}, which describes the hom-sets in the homotopy category of a category of fibrant objects. Indeed, we will derive a closely related result as a corollary.

\begin{thm}
\label{thm:hocolim.formula.for.hammock.spaces.for.simplicial.categories.of.fibrant.objects}
Let $\ul{\mathcal{C}}$ be a simplicial category of fibrant objects, let $\ul{\LH \mathcal{C}}$ be the hammock localisation, let $X$ be an object in $\mathcal{C}$, let $\ul{\mathcal{Q}}$ be the full subcategory of the simplicially enriched slice category $\overcat{\ul{\mathcal{C}}}{X}$ spanned by the trivial fibrations, and let $\ul{U} : \ul{\mathcal{Q}} \to \ul{\mathcal{C}}$ be the evident projection. Then,
\[
\textstyle \hocolim[\op{\ul{\mathcal{Q}}}] \ulHom[\mathcal{C}]{\ul{U}}{\blank} \simeq \ulHom[\LH \mathcal{C}]{X}{\blank}
\]
by a zigzag of weak equivalences of functors $\mathcal{C} \to \cat{\SSet}$. In particular,
\[
\textstyle \hocolim[\op{\ul{\mathcal{Q}}}] \ulHom[\mathcal{C}]{\ul{U}}{\blank} : \mathcal{C} \to \cat{\SSet}
\]
preserves weak equivalences.
\end{thm}
\begin{proof}
Let $\mathcal{Q}$ be the underlying ordinary category of $\ul{\mathcal{Q}}$. By lemmas~\ref{lem:tensors.and.homotopy.cofinality} and~\ref{lem:category.of.trivial.fibrations.is.cotensored},
\[
\textstyle \hocolim[\op{\mathcal{Q}}] \ulHom[\mathcal{C}]{U}{\blank} \simeq \hocolim[\op{\ul{\mathcal{Q}}}] \ulHom[\mathcal{C}]{\ul{U}}{\blank} 
\]
so it suffices to verify the following: 
\[
\textstyle \hocolim[\op{\mathcal{Q}}] \ulHom[\mathcal{C}]{U}{\blank} \simeq \ulHom[\LH \mathcal{C}]{X}{\blank}
\]
Moreover, recalling \autoref{thm:categories.of.fibrant.objects.with.functorial.path.objects.admit.a.homotopical.calculus.of.cocycles} and \autoref{prop:simplicial.categories.of.fibrant.objects.have.functorial.path.objects}, we may apply \autoref{thm:fundamental.theorem.of.homotopical.two-arrow.calculi} and \autoref{lem:replacing.cocycles.with.special.cocycles} to reduce the problem to showing that
\[
\textstyle \hocolim[\op{\mathcal{Q}}] \ulHom[\mathcal{C}]{U}{\blank} \simeq \nv{\Cocyc[\mathcal{C}][\mathcal{V}]{X}{\blank}}
\]
by a zigzag of weak equivalences of functors. By using Thomason's homotopy colimit theorem (\ref{thm:Thomason.hocolim}), it is not hard to see that there is a weak equivalence 
\[
\textstyle \hocolim[\op{\mathcal{Q}}] \discr \Hom[\mathcal{C}]{U}{\blank} \simeq \nv{\Cocyc[\mathcal{C}][\mathcal{V}]{X}{\blank}}
\]
of functors $\mathcal{C} \to \cat{\SSet}$, where on the LHS we have the \emph{ordinary} hom-functor. In particular,
\[
\textstyle \hocolim[\op{\cat{\Simplex}}] \hocolim[\op{\mathcal{Q}}] \discr \Hom[\mathcal{C}]{U}{\Delta^{\bullet} \cotens \pblank} \simeq \hocolim[\op{\cat{\Simplex}}] \nv{\Cocyc[\mathcal{C}][\mathcal{V}]{X}{\Delta^{\bullet} \cotens \pblank}}
\]
but on the one hand, by \autoref{cor:action.of.weak.equivalences.on.two-arrow.zigzags},
\[
\textstyle \nv{\Cocyc[\mathcal{C}][\mathcal{V}]{X}{\blank}} \simeq \hocolim[\op{\cat{\Simplex}}] \nv{\Cocyc[\mathcal{C}][\mathcal{V}]{X}{\Delta^{\bullet} \cotens \pblank}}
\]
and on the other hand, by the Bousfield--Kan theorem,\footnote{See paragraph 4.3 in \citep[\Chap XII]{Bousfield-Kan:1972} or Theorem 18.7.4 in \citep{Hirschhorn:2003}.}
\[
\textstyle \hocolim[\op{\cat{\Simplex}}] \discr \Hom[\mathcal{C}]{U}{\Delta^{\bullet} \cotens \blank} \simeq \ulHom[\mathcal{C}]{U}{\blank} 
\]
so by interchanging homotopy colimits, the claim follows.
\end{proof}

\begin{cor}
With notation as above,
\[
\textstyle \indlim_{\op{\pi_0 \argb{\ul{\mathcal{Q}}}}} \pi_0 \ulHom[\mathcal{C}]{\ul{U}}{\blank} \cong \Hom[\Ho \mathcal{C}]{X}{\blank}
\]
as functors $\mathcal{C} \to \cat{\Set}$.
\end{cor}
\begin{proof}
Since $\pi_0 : \cat{\SSet} \to \cat{\Set}$ is a simplicially enriched left Quillen functor, it takes homotopy colimits in $\cat{\SSet}$ to homotopy colimits in $\cat{\Set}$. Homotopy colimits in $\cat{\Set}$ are the same as (simplicially enriched) colimits, thus,
\[
\textstyle \indlim_{\op{\ul{\mathcal{Q}}}} \pi_0 \ulHom[\mathcal{C}]{\ul{U}}{\blank} \cong \pi_0 \ulHom[\LH \mathcal{C}]{X}{\blank} \cong \Hom[\Ho \mathcal{C}]{X}{\blank}
\]
But the evident localising functor $\ul{\mathcal{Q}} \to \pi_0 \argb{\ul{\mathcal{Q}}}$ induces an equivalence between the category of simplicially enriched diagrams $\op{\ul{\mathcal{Q}}} \to \cat{\Set}$ and the category of (ordinary) diagrams $\op{\pi_0 \argb{\ul{\mathcal{Q}}}} \to \cat{\Set}$, so 
\[
\textstyle \indlim_{\op{\pi_0 \argb{\ul{\mathcal{Q}}}}} \pi_0 \ulHom[\mathcal{C}]{\ul{U}}{\blank} \cong \indlim_{\op{\ul{\mathcal{Q}}}} \pi_0 \ulHom[\mathcal{C}]{\ul{U}}{\blank}
\]
and we are done.
\end{proof}

\section{The Verdier hypercovering theorem}

Throughout this section, let $\mathcal{C}$ be a small category with a Grothendieck topology $J$, let $\mathcal{M}$ be the category of simplicial presheaves on $\mathcal{C}$, equipped with the $J$-local model structure of \citet{Joyal-to-Grothendieck} and \citet{Jardine:1987}, and for each regular cardinal $\kappa$, let $\mathcal{M}_{< \kappa}$ be the full subcategory of $\kappa$-presentable objects in $\mathcal{M}$.

\begin{prop}
\label{prop:model.category.of.small.simplicial.presheaves}
There are arbitrarily large regular cardinals $\kappa$ such that $\mathcal{M}_{< \kappa}$ inherits the structure of a simplicial closed model category from $\mathcal{M}$, including functorial factorisations.
\end{prop}
\begin{proof}
Use Propositions~1.17, 5.9, and~5.20 in \citep{Low:2014a}.
\end{proof}

Recall also the notion of a $J$-local fibration of simplicial presheaves on $\mathcal{C}$: in the case where $\tuple{\mathcal{C}, J}$ is a site with enough points, a morphism of simplicial presheaves on $\mathcal{C}$ is a $J$-local fibration if and only if all its stalks are Kan fibrations. Let $\mathcal{E}$ be the full subcategory of $\mathcal{M}$ spanned by the $J$-locally fibrant simplicial presheaves on $\mathcal{C}$ and let $\mathcal{E}_{< \kappa} = \mathcal{E} \cap \mathcal{M}_{< \kappa}$.

\begin{prop}
\label{prop:category.of.small.locally.fibrant.simplicial.presheaves}
There are arbitrarily large regular cardinals $\kappa$ such that $\ul{\mathcal{E}_{< \kappa}}$ is a simplicial category of fibrant objects, with weak equivalences being the $J$-local weak equivalences and fibrations being the $J$-local fibrations.
\end{prop}
\begin{proof}
Let $\kappa$ be any uncountable regular cardinal such that $\card{\mor \mathcal{C}} < \kappa$. It is clear that axioms A, B, and E are satisfied, and a cardinality argument can be used to verify axiom C and C$\sb{\Delta}$. (Under our hypothesis on $\kappa$, a simplicial presheaf on $\mathcal{C}$ is in $\mathcal{M}_{< \kappa}$ if and only if it has $< \kappa$ elements.) A similar argument shows that the cotensor products $K \cotens X$ are in $\mathcal{M}_{< \kappa}$ if $K$ is a finite simplicial set and $X$ is in $\mathcal{M}_{< \kappa}$, so it suffices to verify that $\mathcal{E}$ satisfies axioms F and G; for this, we may use the same method as the proof of Lemma 1.15 in \citep{Low:2014b}, \ie first reduce to the case of simplicial sheaves, and then apply Barr's embedding theorem to reduce to the case of simplicial sets, which is well known.\footnote{See \eg Theorem 3.3.1 in \citep{Hovey:1999}.}
\end{proof}

Henceforth, fix an infinite regular cardinal $\kappa$ such that $\mathcal{M}_{< \kappa}$ and $\mathcal{E}_{< \kappa}$ satisfy the conclusions of propositions~\ref{prop:model.category.of.small.simplicial.presheaves} and~\ref{prop:category.of.small.locally.fibrant.simplicial.presheaves}. 

\begin{lem}
\label{lem:right.derived.sections.functor}
Let $\Gamma : \op{\mathcal{C}} \times \mathcal{M} \to \cat{\SSet}$ be the functor defined by the following formula:
\[
\Gamma \argp{C, X} = X \argp{C}
\]
Then, for each object $C$ in $\mathcal{C}$:
\begin{enumerate}[(i)]
\item Let $\homs{C}$ be the simplicial presheaf represented by $C$. There is an isomorphism
\[
\Gamma \argp{C, \blank} \cong \ulHom[\mathcal{M}]{\homs{C}}{\blank}
\]
of functors $\mathcal{M} \to \cat{\SSet}$, where the RHS is the simplicial hom-functor.

\item $\Gamma \argp{C, \blank} : \mathcal{M} \to \cat{\SSet}$ is a right Quillen functor.  In particular, it has a total right derived functor $\totalR \Gamma \argp{C, \blank} : \Ho \mathcal{M} \to \Ho \cat{\SSet}$.

\item Let $\LH \mathcal{M}_{< \kappa}$ be the hammock localisation of $\mathcal{M}_{< \kappa}$. There is an isomorphism
\[
\totalR \Gamma \argp{C, \blank} \cong \ulHom[\LH \mathcal{M}_{< \kappa}]{\homs{C}}{\blank}
\]
of functors $\Ho \mathcal{M}_{< \kappa} \to \Ho \cat{\SSet}$.
\end{enumerate}
\end{lem}
\begin{proof}
(i). Use the Yoneda lemma.

\bigskip\noindent
(ii). $\ulHom[\mathcal{M}]{\homs{C}}{\blank}$ is a right Quillen functor because $\mathcal{M}$ is a simplicial closed model category where all objects are cofibrant, so $\Gamma \argp{C, \blank}$ is also a right Quillen functor. The existence of a total right derived functor is then a standard result.\footnote{See \eg Theorem 8.5.8 in \citep{Hirschhorn:2003}.}

\bigskip\noindent
(iii). Apply either Remark 5.2.10 in \citep{Hovey:1999} or Proposition 16.6.23 in \citep{Hirschhorn:2003} to Theorem 3.8 in \citep{Low:2014c}.
\end{proof}

\begin{lem}
\label{lem:hammock.localisation.of.locally.fibrant.simplicial.presheaves.vs.all.simplicial.presheaves}
Let $X$ and $Y$ be objects in $\mathcal{E}_{< \kappa}$. Then the morphism
\[
\ulHom[\LH \mathcal{E}_{< \kappa}]{X}{Y} \to \ulHom[\LH \mathcal{M}_{< \kappa}]{X}{Y}
\]
induced by the inclusion $\mathcal{E}_{< \kappa} \embedinto \mathcal{M}_{< \kappa}$ is a weak homotopy equivalence of simplicial sets.
\end{lem}
\begin{proof}
Use Proposition 3.5 in \citep{Dwyer-Kan:1980b}.
\end{proof}

The following version of the Verdier hypercovering theorem is due to \citet{Jardine:2012} and \citet{MO:answer-165324}.

\begin{prop}
\label{prop:Jardine-Rezk}
Let $C$ be an object in $\mathcal{C}$ and let $\mathcal{V}$ be the subcategory of $J$-local trivial fibrations in $\mathcal{E}_{< \kappa}$. Then there are isomorphisms
\[
\hspace{-3.0ex}
\totalR \Gamma \argp{C, \blank}
\cong \ulHom[\LH \mathcal{E}_{< \kappa}]{\homs{C}}{\blank}
\cong \nv{\Cocyc[\parens{\mathcal{E}_{< \kappa}}]{\homs{C}}{\blank}}
\cong \nv{\Cocyc[\parens{\mathcal{E}_{< \kappa}}][\mathcal{V}]{\homs{C}}{\blank}}
\hspace{-3.0ex}
\]
of functors $\Ho \mathcal{E}_{< \kappa} \to \Ho \cat{\SSet}$.
\end{prop}
\begin{proof}
Combine \autoref{thm:fundamental.theorem.of.homotopical.two-arrow.calculi} and lemmas~\ref{lem:replacing.cocycles.with.special.cocycles}, \ref{lem:right.derived.sections.functor}, and~\ref{lem:hammock.localisation.of.locally.fibrant.simplicial.presheaves.vs.all.simplicial.presheaves}.
\end{proof}

One may then derive a homotopy colimit formula for $\totalR \Gamma \argp{C, \blank}$  analogous to Verdier's original colimit formula (\confer Théorème 7.4.1 in \citep[Exposé~V]{SGA4b} or Theorem 8.16 in \citep{Artin-Mazur:1969}). A similar result was previously obtained by \citet{Jardine:2015}: see Theorem~6.17 and Corollary~6.19 in \opcit.

\begin{prop}
\label{prop:non-abelian.Verdier.hypercovering.formula}
Let $C$ be an object in $\mathcal{C}$, let $\ul{\mathcal{Q}}$ be the simplicially enriched full subcategory of the simplicially enriched slice category $\overcat{\parens{\ul{\mathcal{E_{< \kappa}}}}}{\homs{C}}$ spanned by the $J$-local trivial fibrations, and let $\ul{U} : \ul{\mathcal{Q}} \to \ul{\mathcal{E}_{< \kappa}}$ be the evident projection functor. Then there is an isomorphism
\[
\textstyle \totalR \Gamma \argp{C, \blank} \cong \hocolim[\op{\ul{\mathcal{Q}}}] \ulHom[\mathcal{E}_{< \kappa}]{\ul{U}}{\blank}
\]
of functors $\Ho \mathcal{E}_{< \kappa} \to \Ho \cat{\SSet}$.
\end{prop}
\begin{proof}
Apply \autoref{thm:hocolim.formula.for.hammock.spaces.for.simplicial.categories.of.fibrant.objects} and \autoref{prop:Jardine-Rezk}.
\end{proof}

\begin{cor}
Let $K$ be a Kan complex of cardinality $< \kappa$ and let $\Delta K$ be the constant simplicial presheaf on $\mathcal{C}$ with value $K$. With other notation as above, we have
\[
\textstyle \totalR \Gamma \argp{C, \Delta K} \cong \hocolim[\op{\ul{\mathcal{Q}}}] \ulHom[\SSet]{{\indlim_{\op{\mathcal{C}}}} \circ \ul{U}}{K}
\] 
as objects in $\Ho \cat{\SSet}$, and this is natural in $K$.
\end{cor}
\begin{proof}
As usual, we have the following isomorphism of simplicially enriched functors $\op{\ul{\mathcal{Q}}} \to \ul{\cat{\SSet}}$:
\[
\textstyle \ulHom[\SSet]{{\indlim_{\op{\mathcal{C}}}} \circ \ul{U}}{K} \cong \ulHom[\mathcal{E}_{< \kappa}]{\ul{U}}{\Delta K}
\]
The claim follows, by \autoref{prop:non-abelian.Verdier.hypercovering.formula}.
\end{proof}

\appendix

\section{Categories of fibrant objects redux}

The following is what \citet{Cisinski:2010a} calls a `catégorie dérivable à gauche':

\begin{dfn}
A \strong{Cisinski fibration category} is a category $\mathcal{C}$ equipped with a pair $\tuple{\mathcal{W}, \mathcal{F}}$ of subclasses of $\mor \mathcal{C}$ satisfying these axioms:
\begin{itemize}
\item[\textbf{\LiningNumbers D0.}] $\mathcal{C}$ has a terminal object $1$. A \strong{fibrant object} in $\mathcal{C}$ is an object $X$ such that the unique morphism $X \to 1$ in $\mathcal{C}$ is in $\mathcal{F}$. Any object isomorphic to a fibrant object is fibrant, and $1$ is fibrant.

\item[\textbf{\LiningNumbers D1.}] $\tuple{\mathcal{C}, \mathcal{W}}$ is a category with weak equivalences.

\item[\textbf{\LiningNumbers D2.}] $\mathcal{F}$ is closed under composition and every isomorphism between fibrant objects in $\mathcal{C}$ is in $\mathcal{F}$. If $p : X \to Y$ is in $\mathcal{F}$ and $g : Y' \to Y$ a morphism between fibrant objects in $\mathcal{C}$, then the pullback of $p$ along $g$ exists in $\mathcal{C}$ and is a morphism that is in $\mathcal{F}$.

\item[\textbf{\LiningNumbers D3.}] If $p : X \to Y$ is in $\mathcal{W} \cap \mathcal{F}$ and $g : Y' \to Y$ is a morphism between fibrant objects in $\mathcal{C}$, then the pullback of $p$ along $f$ (exists in $\mathcal{C}$ and) is a morphism that is in $\mathcal{W} \cap \mathcal{F}$.

\item[\textbf{\LiningNumbers D4.}] If $f : X \to Y$ is a morphism in $\mathcal{C}$ and $Y$ is fibrant, then there exist a morphism $i : X \to \hat{X}$ in $\mathcal{W}$ and a morphism $p : \hat{X} \to Y$ in $\mathcal{F}$ such that $f = p \circ i$.
\end{itemize}
In a Cisinski fibration category as above,
\begin{itemize}
\item a \strong{weak equivalence} is a morphism in $\mathcal{W}$,

\item a \strong{fibration} is a morphism in $\mathcal{F}$, and

\item a \strong{trivial fibration} (or \strong{acyclic fibration}) is a morphism in $\mathcal{W} \cap \mathcal{F}$.
\end{itemize}
We will often abuse notation and say $\mathcal{C}$ is a Cisinski fibration category, without mentioning the data $\mathcal{W}$ and $\mathcal{F}$.
\end{dfn}

\begin{example}
Every category of fibrant objects is a Cisinski fibration category in the obvious way. Moreover, if $\mathcal{C}$ is a category of fibrant objects and $Y$ is an object in $\mathcal{C}$, then the slice category $\overcat{\mathcal{C}}{Y}$ is Cisinski fibration category where the fibrant objects are the fibrations with codomain $Y$.
\end{example}

The following is essentially the statement that every object in a Cisinski fibration category can be replaced with a fibrant object in a homotopically unique way. The proof is due to Denis-Charles Cisinski and also appears in \citep{BHH:2015}; we thank Geoffroy Horel for sharing it with us.

\begin{thm}[Cisinski]
\label{thm:uniqueness.of.fibrant.replacements}
Let $\mathcal{C}$ be a Cisinski fibration category and let $\intr{\mathcal{C}}$ be the full subcategory of $\mathcal{C}$ spanned by the fibrant objects. Then the inclusion $\weq \intr{\mathcal{C}} \embedinto \weq \mathcal{C}$ is a homotopy cofinal functor.
\end{thm}
\begin{proof}
Let $X$ be an object in $\mathcal{C}$. We must show that the comma category $\commacat{X}{U}$ is weakly contractible. By the asphericity lemma (1.6) in \citep{Cisinski:2010b}, it suffices to verify the following: for any finite poset $\mathcal{J}$ and any diagram $F : \mathcal{J} \to \commacat{X}{\weq \intr{\mathcal{C}}}$, there is a zigzag of natural transformations connecting $F$ to a constant diagram.

First, observe that diagrams $F : \mathcal{J} \to \commacat{X}{\weq \intr{\mathcal{C}}}$ are the same as diagrams functors $Y : \mathcal{J} \to \weq \mathcal{C}$ equipped with a cone $\phi : \Delta X \hoto Y$. By Théorème~1.30 in \citep{Cisinski:2010a}, there is a natural weak equivalence $\theta : Y \hoto \hat{Y}$ where $\hat{Y}$ is fibrant over the boundaries (`fibrant sur les bords'), so by Proposition~1.18 in \opcit, the limit $\prolim_\mathcal{J} \hat{Y}$ exists in $\mathcal{C}$ and fibrant. Thus, the cone $\theta \bcirc \phi : \Delta X \hoto \hat{Y}$ can be factored as a weak equivalence $i : X \to \hat{X}$ in $\mathcal{C}$ followed by a cone $\hat{\phi} : \Delta \hat{X} \hoto \hat{Y}$, where $\hat{X}$ is a fibrant object in $\mathcal{C}$, and hence, we have the following diagram in $\Func{\mathcal{J}}{\weq \mathcal{C}}$:
\[
\begin{tikzcd}
\Delta X \dar[swap]{\phi} \rar[equals] &
\Delta X \dar \rar[equals] &
\Delta X \dar{\Delta i} \\
Y \rar[swap]{\theta} &
\hat{Y} \rar[swap, leftarrow]{\hat{\phi}} &
\Delta \hat{X}
\end{tikzcd}
\]
This shows that $F : \mathcal{J} \to \commacat{X}{\weq \intr{\mathcal{C}}}$ is indeed connected to a constant diagram.
\end{proof}

\begin{cor}
\label{cor:homotopical.uniqueness.of.correspondence.replacements}
Let $\mathcal{C}$ be a category of fibrant objects, let $\tuple{X, Y}$ be a pair of objects in $\mathcal{C}$, and let $\FCorr[\mathcal{C}]{X}{Y}$ be the category of functional correspondences $X \profto Y$. Then the inclusion $U_{X, Y} : \FCorr[\mathcal{C}]{X}{Y} \embedinto \Cocyc[\mathcal{C}]{X}{Y}$ is homotopy cofinal.
\end{cor}
\begin{proof}
Let $\tuple{f, w} : X \profto Y$ be a cocycle in $\mathcal{C}$. We must show that the comma category $\commacat{\tuple{f, w}}{U_{X, Y}}$ is weakly contractible. Let $\mathcal{D}$ be the slice category $\overcat{\mathcal{C}}{Y \times X}$ considered as a Cisinski fibration category in the obvious way. It is not hard to see that $\FCorr[\mathcal{C}]{X}{Y}$ is isomorphic to a full subcategory of $\weq \overcat{\mathcal{C}}{Y \times X}$, contained in the full subcategory $\weq \intr{\parens{\overcat{\mathcal{C}}{Y \times X}}}$ spanned by the fibrant objects (\ie fibrations in $\mathcal{C}$ with codomain $Y \times X$). Moreover, the comma category $\commacat{\tuple{f, w}}{U_{X, Y}}$ is isomorphic to the comma category $\commacat{\prodtuple{f, w}}{\weq \intr{\parens{\overcat{\mathcal{C}}{Y \times X}}}}$, so by \autoref{thm:uniqueness.of.fibrant.replacements}, $\commacat{\tuple{f, w}}{U_{X, Y}}$ is weakly contractible.
\end{proof}

\needspace{3\baselineskip}
As promised, we obtain the following generalisation of \autoref{thm:categories.of.fibrant.objects.with.functorial.path.objects.admit.a.homotopical.calculus.of.cocycles}:

\begin{thm}
\label{thm:categories.of.fibrant.objects.admit.a.homotopical.calculus.of.cocycles}
Let $\mathcal{C}$ be a category of fibrant objects, let $\mathcal{V}$ be the subcategory of trivial fibrations in $\mathcal{C}$, and let $\FCorr[\mathcal{C}]$ be the category of all functional correspondences in $\mathcal{C}$. Then $\mathcal{V} \subseteq \mor \weq \mathcal{C}$, $\FCorr[\mathcal{C}]$, and $\FCorr[\mathcal{C}] \embedinto \ZzCat{\bracket{-1 ; 1}}{\mathcal{C}}$ constitute a homotopical calculus of cocycles in $\mathcal{C}$.
\end{thm}
\begin{proof}
Combine \autoref{lem:Grothendieck.fibration.of.functional.correspondences} and \autoref{cor:homotopical.uniqueness.of.correspondence.replacements}.
\end{proof}

\begin{remark}
Building on the above, \citet{Meier:2015} showed that any category of fibrant objects that is a homotopical category is itself a fibrant object in the model category of relative categories of \citet{Barwick-Kan:2012a}.
\end{remark}